\documentclass[a4paper,11pt,reqno]{amsart} 
\usepackage{amssymb,amsthm,amsmath}
\usepackage{amsmath}
 \usepackage{enumerate}
\usepackage{ifthen}
\usepackage{graphicx}
\usepackage{cite}
\usepackage{tikz} 
\usetikzlibrary{arrows.meta}
\usepackage{comment}
\usepackage{amsmath,amscd,amssymb}
\usepackage{latexsym}
\usepackage[colorlinks,citecolor=blue,pagebackref,hypertexnames=false]{hyperref}

\numberwithin{equation}{section}
\allowdisplaybreaks[2]
\theoremstyle{plain}
\newtheorem{theorem}{Theorem}[section]

\newtheorem{lemma}[theorem]{Lemma}

\newtheorem{proposition}[theorem]{Proposition}

\theoremstyle{definition}
\newtheorem{definition}[theorem]{Definition}

\theoremstyle{remark}
\newtheorem{remark}[theorem]{Remark}

\newtheorem{case[theorem]}{Case}

\def\norm#1.#2.{\lVert#1\rVert_{#2}}

\def\Hd{\mathcal{H}_{k,l}}

\title[Phase-Space Analysis of Anharmonic and Ornstein-Uhlenbeck Semigroups]{Phase-Space Analysis of generalised Fractional Anharmonic and Ornstein-Uhlenbeck Semigroups on  Weighted Modulation Spaces}

\author{Aparajita Dasgupta}
\author{Uttam Kumar Dolai}

 \address{\endgraf Department of Mathematics
 Indian Institute of Technology, Delhi, Hauz Khas
New Delhi-110016
 India}
\email{adasgupta@maths.iitd.ac.in}
\address{\endgraf Department of Mathematics
Indian Institute of Technology, Delhi, Hauz Khas
New Delhi-110016
India}
 \email{maz248176@maths.iitd.ac.in}


\date{\today}
\begin{document}

\begin{abstract}

We develop a phase-space framework for fractional generalised anharmonic oscillators and their heat semigroups on weighted modulation spaces. We consider operators of the form
\[
\Hd=(-\Delta)^{l}+V(x),
\]
where $V$ is a strictly positive homogeneous potential of polynomial growth of order $2k$. By studying a H\"ormander metric adapted to the quasi-homogeneous symbol $|\xi|^{2l}+V(x)$, as in \cite{MR4299820, MR4944933} we place $\Hd$ and its fractional powers within the Weyl-H\"ormander calculus. In this setting, we show that the fractional operators $\Hd^{\beta}$, $\beta>0$, are globally hypoelliptic pseudodifferential operators and derive refined symbol estimates for the heat semigroup $e^{-t\Hd^{\beta}}$.
These estimates yield boundedness and smoothing properties of the fractional anharmonic heat semigroup on weighted modulation spaces $\mathcal{M}^{p,q}_{s}$, for the full range $0<p,q\leq\infty$ and suitable range of $s$. As applications, we establish global well-posedness of nonlinear heat equations associated with $\Hd^{\beta}$, including both homogenous power and spatially inhomogenous nonlinearities. Finally, we introduce Gaussian modulation spaces adapted to the Ornstein-Uhlenbeck operator and prove continuity of the corresponding semigroup, providing a phase-space perspective complementary to classical Gaussian harmonic analysis.
\end{abstract}

\maketitle
\allowdisplaybreaks \vspace{-.5cm}
\tableofcontents 	

\section{Introduction}

The operators
\[
\mathcal{A}_{k,l}=(-\Delta)^{l}+|x|^{2k},
\]
commonly referred to as anharmonic oscillators, arise naturally in quantum mechanical models describing systems in which purely harmonic behavior is no longer an adequate approximation. Such operators appear, for instance, in the analysis of vibrational dynamics of molecular systems where quadratic confinement fails to capture essential features of the energy spectrum. In these situations, higher-order polynomial potentials provide a more faithful representation of the underlying quantum dynamics, leading to models with richer spectral and dynamical properties.

From a mathematical standpoint, the study of anharmonic oscillators has a long and well-developed history, particularly in connection with spectral theory. For polynomial potentials of even degree, sharp asymptotic formulas for eigenvalues and detailed descriptions of spectral properties were obtained by Helffer and Robert \cite{MR657970,MR683006,MR662451}, with further refinements and extensions in \cite{MR743094,MR2304165,MR2599384}. More recently, attention has turned to one-dimensional anharmonic oscillators with nonsmooth potentials \cite{MR3554704}, revealing additional analytical challenges beyond the classical smooth setting. A significant generalisation was achieved in \cite{MR4489248}, where operators of the form $A(D)+V(x)$, with smooth symbols $A(\xi)$ and $V(x)$ satisfying suitable growth conditions, were analysed within the Weyl-H\"ormander calculus. The generalised anharmonic oscillators considered in the present work fall naturally within this class, and the corresponding H\"ormander metric and symbolic calculus can be constructed in an analogous manner. Moreover, several operators previously studied in the literature, including those in \cite{MR4299820,MR4944933}, arise as special instances of our framework.

The Schr\"odinger and heat semigroups generated by anharmonic oscillators have been widely investigated and play an important role in mathematical physics and partial differential equations; see, for instance, \cite{MR883643, liverts2008approximate, MR1681103,MR752806, MR2885963}. In particular, phase-space properties of the Hermite semigroup were studied in \cite{MR4313961}, and analogous estimates for fractional harmonic oscillators were obtained in \cite{MR4673072}. These results were extended in \cite{MR4944933} to special anharmonic oscillators of the form $(-\Delta)^l+|x|^{2k}$, still within the framework of unweighted modulation spaces $\mathcal{M}^{p,q}$.
Compared to \cite{MR4313961,MR4673072,MR4944933}, where only the harmonic and anharmonic oscillator or the model potential $|x|^{2k}$ is treated and boundedness is obtained exclusively in unweighted modulation spaces, we consider general polynomial-type potentials $V(x)$ and establish boundedness of the semigroup $e^{-t\Hd^{\beta}}$ in weighted modulation spaces $\mathcal{M}^{p,q}_{s}$.

Despite the extensive spectral analysis in \cite{MR4489248}, much less is known about the behavior of the associated semigroups on function spaces adapted to phase-space and time-frequency analysis. In particular, existing works do not address the mapping properties of heat semigroups generated by fractional powers of generalised anharmonic oscillators on weighted modulation spaces. Questions concerning boundedness, smoothing effects, and continuity in these spaces remain largely open within the general H\"ormander metric framework, even though such properties are central to the study of evolution equations and nonlinear problems.

Motivated by this gap, we investigate the positive fractional powers $\Hd^{\beta}$, $\beta>0$, of the generalised anharmonic oscillator
\[
\Hd=(-\Delta)^{l}+V(x),
\]
where $k,l\ge1$ and $V(x)$ is a strictly positive homogeneous potential of polynomial growth of order $2k$. Our approach relies on the construction of a H\"ormander metric $g=g^{k,l}$ adapted to the geometry of the combined symbol $|\xi|^{2l}+V(x)$, together with suitable modulation weight functions. Within this framework, we establish boundedness and decay properties for the semigroup $e^{-t\Hd^{\beta}}$ on weighted modulation spaces $\mathcal{M}^{p,q}_{s}$, $s\in\mathbb{R}$, for the full range $0<p,q\leq\infty$.

Among the extensive literature on modulation spaces, we refer to \cite{MR1843717}
for historical perspectives and the motivations underlying their introduction.
Various characterisations and fundamental properties of modulation spaces have
been developed in
\cite{ CorderoNicola09, Toft04, feichtinger1983modulation, WangHudzik07, SugimotoTomita07}.
The boundedness of a wide class of operators on modulation spaces has been
investigated in
\cite{feichtinger1983modulation, FeichtingerNarimani10, MR3636061,
RodinoWahlberg12, WangZhao09, Wahlberg10, BenyiOh12},
with applications to Fourier integral operators discussed in \cite{CorderoNicola07}.
Furthermore, modulation spaces have been proved to be a natural framework for the
analysis of nonlinear evolution equations; see
\cite{Boulkhemair98, CunananOkoudjou11, GuoFanWu11,
MR2668420,  SugimotoTomita09, MR2506839,MR4286055,MR4291583,MR4201879,MR2028532,MR1843717,MR3014810}. Their algebraic properties and applications to pseudodifferential operators and nonlinear equations were developed in \cite{MR2204673,MR4849356,MR2028532,MR3014810, MR3636061}. In particular, nonlinear heat and Schr\"odinger equations in modulation spaces were investigated in \cite{MR2506839,MR4313961,MR4944933,MR4291583,MR4215324}.

Our work develops a unified phase-space framework that simultaneously accommodates a broad class of generalised anharmonic oscillators and weighted modulation spaces, thereby substantially extending the scope of existing results. Weighted modulation spaces provide a natural refinement of modulation spaces by allowing precise control of decay and regularity in the phase-space. The incorporation of moderate weights is essential in this context, as moderateness guarantees invariance under time–frequency shifts and provides the convolution estimates required for the boundedness of the short-time Fourier transform and related operators (see \cite{MR1843717}). In particular, the modulation spaces considered in \cite{MR4313961,MR4673072,MR4944933} arise as special cases of the weighted modulation spaces developed in this work. From this viewpoint, weighted modulation spaces provide a finer description of phase-space regularity, in which both the polynomial growth of the potential and the anisotropic behavior of the associated symbol are simultaneously encoded. In this sense, both the operators and the functional framework employed in \cite{MR4313961,MR4673072,MR4944933} are recovered as special instances of the present theory.

Within this framework, our analysis applies to a wide class of potentials, including the harmonic oscillator $V(x)=|x|^{2}$, higher-order anharmonic models $V(x)=|x|^{2k}$, anisotropic potentials of the form
\[
V(x)=\sum_{j=1}^{d} a_j |x_j|^{2k}, \qquad a_j>0,
\]
as well as more general positive homogeneous polynomials of degree $2k$. Consequently, the framework developed here captures a wide range of anisotropic and non-radial dynamics that fall outside the reach of previously available modulation-space approaches, and provides a flexible analytical setting for further investigations of both linear and nonlinear problems associated with anharmonic-type operators.
By constructing a H\"ormander metric adapted to the quasi-homogeneous structure of the symbol $|\xi|^{2l}+V(x),$ of $\Hd,$ and applying Weyl-H\"ormander calculus \cite{MR1011988,MR2304165,MR2599384,MR2668420,MR3636061}, we show that $\Hd^{\beta}$ can be realised as a globally hypoelliptic pseudodifferential operator and derive refined symbol estimates for the associated heat semigroup. This provides a new phase-space description of fractional generalised anharmonic oscillators and extends the existing modulation space theory for heat semigroups from special model operators to a broad class of polynomial-type potentials.

Our main object of study is the positive fractional power $\Hd^{\beta}$, with $\beta>0$, and the boundedness properties of the associated heat semigroup $e^{-t\Hd^{\beta}}$ on weighted modulation spaces.

{In this article, we investigate the nonlinear evolution equations with power nonlinearity and spatially inhomogeneous power nonlinearity namely,
\begin{equation}
\label{eq:main}
\begin{cases}
\partial_t u + \Hd^{\beta} u = \lambda |u|^{2\nu}u, \\
u(0,x) = u_0(x),
\end{cases}
\end{equation}
 on $(0,\infty) \times \mathbb{R}^d$, where $\nu \in \mathbb{N}$, $\lambda \in \mathbb{C}$, and $\beta>0$,\\
and
\begin{align}
\label{eq:main1}
\begin{cases}
\partial_{t}u + \Hd^{\beta}u
= \lambda |x|^{-\alpha}\big(|u|^{2\nu}u\big),\\
u(0,x) = u_{0}(x),
\end{cases}
\end{align}
on $\in[0,\infty)\times\mathbb{R}^{d}$, where  $\lambda=\pm1$, $\alpha>0$, and $\nu\in\mathbb{N}$.

Such equations naturally arise in the study of nonlinear diffusion and dispersive phenomena.
Earlier works (see \cite{MR4313961, MR4944933}) considered this type of problems \ref{eq:main} for Hermite and anharmonic operators in the classical modulation spaces $\mathcal{M}^{p,q}$ with $p,q \ge 1$. More recently, it was observed by E.~Cordero \cite{MR4215324} that these equations associated to Hermite operator can also be treated within the broader framework of weighted modulation spaces $\mathcal{M}^{p,q}_s$, even in the case $0<q \le 1$.
Motivated by this observation, we investigate the equation \eqref{eq:main} in the weighted modulation spaces $\mathcal{M}^{p,q}_s$ with $1 \le p,q \le \infty$ and
\[
s >  \frac{d}{q'}.
\]
For this range of $s$, the space $\mathcal{M}^{p,q}_s$ becomes a multiplicative algebra. This algebra property allows us to control the nonlinear term effectively and, in contrast with \cite{MR4313961, MR4944933}, no additional conditions on the parameters $\beta$, $\nu$ are required in order to obtain global-in-time solutions of the equation, as was required in \cite{MR4944933, MR4313961}.  Moreover to the best of our knowledge, there are no existing results addressing the well-posedness of nonlinear heat equations associated with the anharmonic operator involving a spatially power nonlinearity of the form \eqref{eq:main1}.

{
As a special case of \eqref{eq:main1}  we also obtain the global well-posedness result associated with the Hermite operator,
\begin{align*}
H:=-\Delta+|x|^{2},
\end{align*}
which corresponds to the special case $k=l=1$ of the operator $\Hd$. 
For the classical Laplacian, related inhomogeneous nonlinear problems with initial data in modulation spaces have been recently studied in \cite{BHIMANI2026114106}.  
}
 
    The aim of this work is therefore twofold: to analyse the mapping properties of the heat semigroups generated by fractional operators and  to establish well-posedness results for the nonlinear equations in weighted modulation spaces.
Moreover, unlike \cite{MR4313961,MR4673072,MR4944933,MR4215324}, which studied specific model operators and rely on unweighted modulation spaces, the present work treats arbitrary fractional powers $\Hd^{\beta}$ with $\beta>0$ and establishes boundedness results on weighted modulation spaces $\mathcal{M}^{p,q}_{s}$. This combination of fractional powers, general polynomial-type potentials, and weighted modulation spaces has not been previously addressed. The weighted phase-space framework developed here allows one to simultaneously capture the fractional order of the operator and the polynomial growth of the potential that cannot be accommodated within the unweighted settings.

Finally, we study the Ornstein-Uhlenbeck operator
\[
\mathbf{L}=-\Delta+2x\cdot\nabla+d,
\]
which plays a central role in stochastic analysis, semigroup theory, and infinite-dimensional dynamics; see \cite{MR4628746,MR3966424}. When endowed with the Gaussian measure $\gamma_d(x)\,dx=\pi^{-d/2}e^{-|x|^2}dx$, this operator generates the Ornstein-Uhlenbeck semigroup, whose properties have been extensively investigated in Gaussian harmonic analysis, for details we refer \cite{MR3966424}. Despite this extensive literature, the dispersive and phase-space behavior of the Ornstein-Uhlenbeck semigroup from the perspective of partial differential equations has remained largely unexplored.
In this paper, we also address this gap by introducing a new modulation space based on the Gaussian short-time Fourier transform and by proving continuity properties of the Ornstein–Uhlenbeck semigroup in this setting. This yields a phase-space framework that is intrinsically adapted to the Gaussian geometry of the operator and cannot be obtained from the classical modulation spaces used for the Laplacian or the harmonic oscillator. As a result, our approach provides a genuinely new phase-space description of the Ornstein–Uhlenbeck semigroup and extends modulation-space methods to an operator that has so far been studied almost exclusively through Gaussian harmonic analysis rather than through PDE-oriented time–frequency techniques.

The paper is organized as follows. Section~2, presents the necessary background on modulation spaces, symbolic calculus, and the Weyl–Hörmander framework. In Section~3, we establish boundedness and smoothing properties of the heat semigroup generated by fractional generalised anharmonic oscillators on weighted modulation spaces. Section~4, is devoted to global well-posedness results for nonlinear heat equations associated with fractional anharmonic operators, encompassing both power-type and spatially inhomogeneous power nonlinearities. Finally, Section~5, investigates the Ornstein–Uhlenbeck semigroup on Gaussian modulation spaces and establishes continuity results within this adapted phase–space framework.

\section{Preliminaries}

This section briefly recalls some basic facts and properties of modulation spaces and the Weyl-H\"ormander calculus. For the definitions and fundamental properties of modulation spaces, we refer the reader to \cite{ feichtinger1983modulation, MR1843717}  and  details on the Weyl-H\"ormander calculus can be found in the monograph by Nicola–Rodino \cite{MR2668420} among many others. For further discussion of the Weyl–Hörmander calculus in the context of anharmonic oscillators, we refer to \cite{MR4299820, MR4489248} and the references therein.

\subsection{Modulation Spaces}

We begin by recalling the definition of the short-time Fourier transform (STFT).
Let $f\in\mathcal{S}'(\mathbb{R}^d)$ and let $g\in\mathcal{S}(\mathbb{R}^d)\setminus\{0\}$
be a fixed window function. The STFT of $f$ with respect to $g$ is defined by
\[
V_g f(x,\xi)
= \int_{\mathbb{R}^d} f(z)\,\overline{g(z-x)}\,e^{-2\pi i \xi\cdot z}\,\mathrm{d}z,
\qquad (x,\xi)\in\mathbb{R}^{2d}.
\]

Modulation spaces were introduced by Feichtinger in the 1980s \cite{feichtinger1983modulation}
and constitute a fundamental tool in time-frequency analysis.

A weight function $v$ on $\mathbb{R}^{2d}$ is called \emph{submultiplicative} if
\[
v(x+y)\le v(x)\,v(y), \qquad x,y\in\mathbb{R}^{2d}.
\]
A weight $w$ is said to be \emph{$v$-moderate} if
\[
w(x+y)\le v(x)\,w(y), \qquad x,y\in\mathbb{R}^{2d}.
\]

Of particular importance in this article are the polynomial weights
\[
v_s(x,\xi)
:= \bigl(1+|x|^2+|\xi|^2\bigr)^{\frac{s}{2}},
\qquad s\in\mathbb{R}.
\]
Weights that are $v_s$-moderate for some $s\in\mathbb{R}$ are referred to as
\emph{polynomially moderate}.

Let $0<p,q\le\infty$ and let $w$ be a polynomially moderate weight.
The modulation space $\mathcal{M}^{p,q}_w(\mathbb{R}^d)$ consists of all
$f\in\mathcal{S}'(\mathbb{R}^d)$ such that
\[
\|f\|_{\mathcal{M}^{p,q}_w}
:= \|V_g f\, w\|_{L^{p,q}}
= \left(
\int_{\mathbb{R}^d}
\left(
\int_{\mathbb{R}^d}
|V_g f(x,\xi)\,w(x,\xi)|^p\,\mathrm{d}x
\right)^{\frac{q}{p}}
\mathrm{d}\xi
\right)^{\frac{1}{q}}
<\infty,
\]
with the usual modifications when $p=\infty$ or $q=\infty$.
Here $L^{p,q}(\mathbb{R}^d\times\mathbb{R}^d)$ denotes the mixed Lebesgue space
with (quasi-)norm
\[
\|F\|_{L^{p,q}} := \bigl\|\|F(x,\xi)\|_{L_x^p}\bigr\|_{L_\xi^q}.
\]

When the weight is exactly $v_s$, we use the shorthand notation
\[
\mathcal{M}^{p,q}_s := \mathcal{M}^{p,q}_{v_s}.
\]

Modulation spaces are quasi-Banach spaces (and Banach spaces when $p,q\ge1$),
and their definition is independent of the choice of window function $g$,
in the sense that different window functions yield equivalent norms.
We refer to
\cite{MR4286055,MR4201879,MR2028532,MR1843717}
for proofs and further details.

We conclude this subsection by recalling several properties of modulation spaces
that will be used throughout the paper.

\begin{lemma}[\cite{MR4849356}]
\label{Inclusionrelation}
Let $1\le p,p_1,p_2,q,q_1,q_2\le\infty$ and $s,s_1,s_2\in\mathbb{R}$. Then:
\begin{enumerate}
\item If $p,q<\infty$, then $\mathcal{S}(\mathbb{R}^d)$ is dense in
$\mathcal{M}^{p,q}_s(\mathbb{R}^d)$ and
\[
\bigl(\mathcal{M}^{p,q}_s(\mathbb{R}^d)\bigr)'
= \mathcal{M}^{p',q'}_{-s}(\mathbb{R}^d),
\]
where $p'$ and $q'$ denote the conjugate exponents of $p$ and $q$.

\item If $p_1\le p_2$, $q_1\le q_2$, and $s_1\ge s_2$, then
\[
\mathcal{M}^{p_1,q_1}_{s_1}(\mathbb{R}^d)
\hookrightarrow
\mathcal{M}^{p_2,q_2}_{s_2}(\mathbb{R}^d).
\]

\item If $s>\frac{d}{q'}$, or if $q=1$ and $s\ge0$, then
$\mathcal{M}^{p,q}_s(\mathbb{R}^d)\subset C(\mathbb{R}^d)$ and
$\mathcal{M}^{p,q}_s(\mathbb{R}^d)$ is a multiplicative algebra.
More precisely, there exists a constant $C>0$ such that
\[
\|fg\|_{\mathcal{M}^{p,q}_s}
\le C\,
\|f\|_{\mathcal{M}^{p,q}_s}\,
\|g\|_{\mathcal{M}^{p,q}_s},
\qquad
f,g\in\mathcal{M}^{p,q}_s(\mathbb{R}^d).
\]
\end{enumerate}
\end{lemma}

\subsection{A Modulation Weight Adapted to $\Hd$}

Let $k,l\ge1$ be integers and consider the generalised anharmonic oscillator
\[
\mathcal H_{k,l}=(-\Delta)^l+V(x),
\]
where the potential $V(x)$ has polynomial growth of order $|x|^{2k}$.
We define a phase-space weight by
\[
\widetilde v(x,\xi)
:= q_{1}+V(x)^{1/2}+|\xi|^l,
\qquad (x,\xi)\in\mathbb R^d\times\mathbb R^d,
\]
where $q_{1}>0$ is fixed. Since $V(x)^{1/2}\sim |x|^k$ for large $|x|$, this weight reflects the
intrinsic anisotropic scaling of $\mathcal H_{k,l}$.

For all $x,y,\xi,\eta\in\mathbb R^d$, there exists a constant $C>1$ such that
\begin{align}
\widetilde v(x,\xi)\,\widetilde v(y,\eta)
&= (q_{1}+V(x)^{1/2}+|\xi|^l)(q_{1}+V(y)^{1/2}+|\eta|^l) \nonumber\\
&\gtrsim (q_{1}+|x|^k+|\xi|^l)(q_{1}+|y|^k+|\eta|^l) \nonumber\\
&\gtrsim q_{1}+V(x+y)^{1/2}+|\xi+\eta|^l
= \widetilde v(x+y,\xi+\eta).
\label{eq:submultiplicative-weight}
\end{align}
Setting $v:=C\,\widetilde v$ therefore yields a submultiplicative weight on phase-pace.
This property guarantees the stability of the associated modulation spaces under time-frequency
shifts and is essential for the boundedness of pseudodifferential operators arising from the
functional calculus of $\mathcal H_{k,l}$.

Throughout the paper we work with weighted modulation spaces $ \mathcal{M}^{p,q}_{s}$ defined
with respect to the weight $v_s=(q_{1}+V(x)^{1/2}+|\xi|^{l})^{s}$, $s\in\mathbb R$. The choice of $v$ allows us to treat
fractional anharmonic semigroups within a unified phase-space framework, extending beyond
classical harmonic and Gaussian settings.
\subsection{Symbolic Calculus}

In this subsection, we introduce several notions and symbol classes that will be used frequently in the sequel. We refer to \cite{MR2668420} for further details and background.

A positive continuous function $\Phi(x,\xi)$ on $\mathbb{R}^{2d}$ is called a
\emph{sublinear weight} if
\begin{equation}
1 \le \Phi(x,\xi) \lesssim 1+|x|+|\xi|,
\qquad x,\xi\in\mathbb{R}^d.
\end{equation}
It is called a \emph{temperate weight} if there exists $s>0$ such that
\begin{equation}\label{temperate}
\Phi(x+y,\xi+\eta)
\lesssim \Phi(x,\xi)\,(1+|y|+|\eta|)^s,
\qquad x,\xi,y,\eta\in\mathbb{R}^d.
\end{equation}
Products and real powers of temperate weights are again temperate, and one always has
\[
(1+|y|+|\eta|)^{-s}
\lesssim \Phi(y,\eta)
\lesssim (1+|y|+|\eta|)^s.
\]

The weights of interest in this work are of the form
\[
\Phi(x,\xi)
:= q_{1} + V(x) + |\xi|^{2l},
\]
where $q_{1}>0$ is chosen such that $\Phi(x,\xi)\ge 1$, and where $V(x)$ is a strictly positive function with polynomial growth of order $2k$. Setting
\[
k_0 := \max\{k,l\},
\]
we observe that
\[
1 \le \Phi(x,\xi)^{\frac{1}{2k_0}}
\lesssim (1+|x|^{2k}+|\xi|^{2l})^{\frac{1}{2k_0}}
\lesssim 1+|x|+|\xi|,
\]
so that $\Phi^{\frac{1}{2k_{0}}}$ is sublinear.

Moreover, $\Phi$ is a temperate weight. Indeed,
\[
\Phi(x+y,\xi+\eta)
= q_{1} + V(x+y) + |\xi+\eta|^{2l}.
\]
Since
\[
|\xi+\eta|^{2l}
\lesssim |\xi|^{2l} + |\eta|^{2l}
\lesssim \Phi(x,\xi)(1+|\eta|)^{2l},
\]
and since $V(x)$ has polynomial growth of order $2k$,
\[
V(x+y)
\lesssim (1+|x|+|y|)^{2k}
\lesssim V(x)\,(1+|y|)^{2k},
\]
we obtain
\[
\Phi(x+y,\xi+\eta)
\lesssim \Phi(x,\xi)\,(1+|y|+|\eta|)^{2\max\{k,l\}},
\]
which proves the claim. Consequently, $\Phi^s$ is temperate for every $s\in\mathbb{R}$.

\begin{definition}[Symbol class]
\label{Symboldefinition}
Let $\Phi(x,\xi)$ and $\Psi(x,\xi)$ be sublinear, temperate weights, and let
$M(x,\xi)$ be a temperate weight. The symbol class $S(M;\Phi,\Psi)$ consists of
all functions $a\in C^\infty(\mathbb{R}^{2d})$ such that for every
$\alpha,\beta\in\mathbb{N}^d$,
\[
|\partial_\xi^\alpha \partial_x^\beta a(x,\xi)|
\lesssim
M(x,\xi)\,\Psi(x,\xi)^{-|\alpha|}\,\Phi(x,\xi)^{-|\beta|},
\qquad (x,\xi)\in\mathbb{R}^{2d}.
\]
\end{definition}

The family of seminorms
\[
\|a\|_{k,S(M;\Phi,\Psi)}
:=
\sup_{|\alpha|+|\beta|\le k}
\sup_{(x,\xi)\in\mathbb{R}^{2d}}
|\partial_\xi^\alpha \partial_x^\beta a(x,\xi)|
M(x,\xi)^{-1}\Psi(x,\xi)^{|\alpha|}\Phi(x,\xi)^{|\beta|},
\]
with $k\in\mathbb{N}$, endows $S(M;\Phi,\Psi)$ with a Fr\'echet space topology.

\begin{definition}
A symbol $a$ is called \emph{globally elliptic} in $S(M;\Phi,\Psi)$ if
$a\in S(M;\Phi,\Psi)$ and there exists $R>0$ such that
\begin{equation}\label{eqn1}
|a(x,\xi)| \gtrsim M(x,\xi),
\qquad |x|+|\xi|\ge R.
\end{equation}
\end{definition}

\begin{definition}
A symbol $a\in S(M;\Phi,\Psi)$ is called \emph{globally hypoelliptic} if there exists
a temperate weight $M_0(x,\xi)$ and $R>0$ such that
\begin{equation}\label{eqn2}
|a(x,\xi)| \gtrsim M_0(x,\xi),
\qquad |x|+|\xi|\ge R,
\end{equation}
and for all $\alpha,\beta\in\mathbb{N}^d$,
\begin{equation}\label{eqn3}
|\partial_\xi^\alpha \partial_x^\beta a(x,\xi)|
\lesssim
|a(x,\xi)|\,\Psi(x,\xi)^{-|\alpha|}\,\Phi(x,\xi)^{-|\beta|},
\qquad |x|+|\xi|\ge R.
\end{equation}
We denote this class by $\mathrm{Hypo}(M,M_0;\Phi,\Psi)$.
\end{definition}

\begin{remark}
\label{remark}
If \eqref{eqn2} holds with $M_0=M$, then $a$ is elliptic.
Conversely, every elliptic symbol in $S(M;\Phi,\Psi)$ is hypoelliptic, since
\eqref{eqn2} holds trivially and \eqref{eqn3} follows from the defining estimates
of the symbol class.
\end{remark}


\subsection{Weyl-H\"ormander Calculus}

In this subsection, we recall basic definitions and notations from the
Weyl-H\"ormander calculus. We refer the reader to
\cite{MR2304165,MR2599384,MR1011988} for further details.

The Weyl quantisation of a symbol
$a\in\mathcal{S}'(\mathbb{R}^d\times\mathbb{R}^d)$
is defined by
\[
a^{w}(x,D)u(x)
:= \frac{1}{(2\pi)^d}
\int_{\mathbb{R}^d}\int_{\mathbb{R}^d}
e^{i\langle x-y,\xi\rangle}
a\!\left(\frac{x+y}{2},\xi\right)
u(y)\,\mathrm{d}y\,\mathrm{d}\xi,
\qquad u\in\mathcal{S}(\mathbb{R}^d).
\]

More generally, for $t\in\mathbb{R}$, the $t$-quantisation of $a$ is given by
\[
a_t(x,D)u(x)
:= \frac{1}{(2\pi)^d}
\int_{\mathbb{R}^d}\int_{\mathbb{R}^d}
e^{i\langle x-y,\xi\rangle}
a(tx+(1-t)y,\xi)
u(y)\,\mathrm{d}y\,\mathrm{d}\xi.
\]
In particular, the Weyl quantisation corresponds to $t=\frac12$, that is,
$a_{\frac12}(x,D)=a^{w}(x,D)$, while the choice $t=1$ yields the
Kohn-Nirenberg quantisation,
\[
a(x,D)u(x)
:= a_1(x,D)u(x)
= \frac{1}{(2\pi)^d}
\int_{\mathbb{R}^d}
e^{i\langle x,\xi\rangle}
a(x,\xi)\,\widehat{u}(\xi)\,\mathrm{d}\xi.
\]

We now recall the notion of a H\"ormander metric, which will play a crucial
role in the analysis of the generalised anharmonic oscillator and its
fractional powers.

\begin{definition}[H\"ormander metric]\label{HM}
Let $X\in\mathbb{R}^{2d}$ and let $g_X(\cdot)$ be a positive definite quadratic
form on $\mathbb{R}^{2d}$. The family $g=(g_X)_{X\in\mathbb{R}^{2d}}$ is called a
\emph{H\"ormander metric} if the following conditions hold:
\begin{enumerate}
\item[\rm (I)] \textbf{Continuity (slowness).}
There exists $C>0$ such that
\[
g_X(X-Y)\le C^{-1}
\quad\Longrightarrow\quad
\left(\frac{g_X(T)}{g_Y(T)}\right)^{\pm1}\le C,
\]
for all $T\in\mathbb{R}^{2d}\setminus\{0\}$.

\item[\rm (II)] \textbf{Uncertainty principle.}
Let $\sigma(Y,Z)=\langle z,\eta\rangle-\langle y,\zeta\rangle$
denotes the symplectic form and define
\[
g_X^{\sigma}(T)
:= \sup_{W\neq0}\frac{\sigma(T,W)^2}{g_X(W)}.
\]
We require that
\[
\lambda_g(X)
:= \inf_{T\neq0}\left(\frac{g_X^{\sigma}(T)}{g_X(T)}\right)^{1/2}
\ge 1,
\qquad X\in\mathbb{R}^{2d}.
\]

\item[\rm (III)] \textbf{Temperateness.}
There exist $\overline{C}>0$ and $J\in\mathbb{N}$ such that
\[
\left(\frac{g_X(T)}{g_Y(T)}\right)^{\pm1}
\le \overline{C}\bigl(1+g_Y^{\sigma}(X-Y)\bigr)^J,
\]
for all $X,Y,T\in\mathbb{R}^{2d}$.
\end{enumerate}
\end{definition}

Associated with a H\"ormander metric $g$ is the \emph{Planck function}
\[
h_g(X)^2
:= \sup_{T\neq0}\frac{g_X(T)}{g_X^{\sigma}(T)},
\]
which satisfies $h_g(X)=\lambda_g(X)^{-1}$.

\begin{definition}[$g$-weight]\label{GW}
Let $M:\mathbb{R}^{2d}\to(0,\infty)$.
\begin{itemize}
\item $M$ is called \emph{$g$-continuous} if there exists $\tilde{C}>0$ such that
\[
g_X(X-Y)\le \tilde{C}^{-1}
\quad\Longrightarrow\quad
\left(\frac{M(X)}{M(Y)}\right)^{\pm1}\le \tilde{C}.
\]

\item $M$ is called \emph{$g$-temperate} if there exists $\tilde{C}>0$ and
$N\in\mathbb{N}$ such that
\[
\left(\frac{M(X)}{M(Y)}\right)^{\pm1}
\le \tilde{C}\bigl(1+g_Y^{\sigma}(X-Y)\bigr)^N.
\]
\end{itemize}
We say that $M$ is a \emph{$g$-weight} if it is both $g$-continuous and
$g$-temperate.
\end{definition}

\begin{definition}
Let $g$ be a H\"ormander metric and let $M$ be a $g$-weight.
The symbol class $S(M,g)$ consists of all functions
$\sigma\in C^\infty(\mathbb{R}^{2d})$ such that for every $k\in\mathbb{N}$
there exists $C_k>0$ with
\begin{equation}\label{inwhk}
|\sigma^{(k)}(X;T_1,\ldots,T_k)|
\le C_k\,M(X)\prod_{j=1}^k g_X^{1/2}(T_j),
\end{equation}
for all $X,T_1,\ldots,T_k\in\mathbb{R}^{2d}$.
\end{definition}

The smallest constant $C_k$ in \eqref{inwhk} defines a seminorm
$\|\sigma\|_{k,S(M,g)}$, and the corresponding family
$\{\|\cdot\|_{k,S(M,g)}\}_{k\in\mathbb{N}}$ endows $S(M,g)$ with a Fr\'echet space
structure.

\subsection{Hörmander Metric Associated with $\mathcal H_{k,l}$}

We now introduce a Hörmander metric intrinsically adapted to the operator $\mathcal H_{k,l}$.
Define
\[
g:=g^{k,l}
= \frac{dx^2}{(q_{1}+V(x)+|\xi|^{2l})^{1/k}}
+ \frac{d\xi^2}{(q_{1}+V(x)+|\xi|^{2l})^{1/l}},
\]
where $q_{1}>0$ is chosen so that $q_{1}+V(x)+|\xi|^{2l}\ge1$ for all
$(x,\xi)\in\mathbb R^d\times\mathbb R^d$.
This metric encodes the quasi-homogeneous interaction between the spatial and frequency
variables.

It was shown in \cite{MR4299820,MR4489248} that $g$ is a Hörmander metric. Moreover, the
function
\[
M(x,\xi):=q_{1}+V(x)+|\xi|^{2l}
\]
is a $g$-weight in the sense of Hörmander's symbolic calculus. Consequently, the symbol class
\[
S(M,g)=S\bigl(q_{1}+V(x)+|\xi|^{2l},\,g\bigr)
\]
is well suited for the analysis of pseudodifferential operators associated with
$\mathcal H_{k,l}$ and its fractional powers.

The corresponding uncertainty parameter is given by
\[
\lambda_g
= (q_{1}+V(x)+|\xi|^{2l})^{\frac{k+l}{2kl}},
\]
and the associated Planck function $h_g$ satisfies
\[
h_g^{-1}=\lambda_g \simeq v^{\frac{k+l}{kl}},
\]
which establishes a precise quantitative link between the phase-space geometry induced by $g$
and the modulation weight $v$.

\subsection{Equivalent Anisotropic Symbol Classes}

An equivalent description of the Hörmander symbol class
$S((q_{1}+V(x)+|\xi|^{2l})^{m/2},g)$ can be given in terms of anisotropic derivative estimates.
A symbol $a\in C^\infty(\mathbb R^d\times\mathbb R^d)$ is said to belong to the class
$\Sigma^m_{k,l}$ if
\begin{align}
\label{symbolclassdef}
|\partial_x^\beta \partial_\xi^\alpha a(x,\xi)|
\le C_{\alpha,\beta}
\bigl(q_{1}+V(x)^{1/2}+|\xi|^l\bigr)^{\,m-\frac{|\beta|}{k}-\frac{|\alpha|}{l}},
\end{align}
for all multi-indices $\alpha,\beta$ and all $(x,\xi)\in\mathbb R^d\times\mathbb R^d$.

By Proposition~4.2 in \cite{MR4489248}, these classes coincide, namely
\[
\Sigma^m_{k,l}
= S\bigl((q_{1}+V(x)+|\xi|^{2l})^{m/2},\,g\bigr).
\]
Equipped with the natural family of seminorms, $\Sigma^m_{k,l}$ is a Fréchet space.


\section{Generalised Fractional Anharmonic Heat Semigroup on Weighted Modulation Spaces}

In this section we establish the main boundedness and smoothing properties of the heat semigroup
generated by fractional powers of the generalised anharmonic oscillator on weighted modulation
spaces. Throughout, we fix integers $k,l\ge1$ and consider the operator
\[
\mathcal H_{k,l}=(-\Delta)^l+V(x),
\]
where $V(x)$ is a strictly positive homogeneous potential of polynomial growth of order $2k$.

As observed in \cite{MR4299820}, the anisotropic symbol classes discussed in Section~2 fit naturally into the general framework $S(M;\Phi,\Psi)$ of Nicola-Rodino \cite{MR2668420}.
Let
\[
k_0 := \max\{k,l\},
\]
the metric $g=g^{(k,l)}$ satisfies
\[
g
\le
\frac{\mathrm{d}x^{2}}{(q_{1}+V(x)+|\xi|^{2l})^{1/k_0}}
+
\frac{\mathrm{d}\xi^{2}}{(q_{1}+V(x)+|\xi|^{2l})^{1/k_0}}.
\]
Define
\[
\Phi(x,\xi)=\Psi(x,\xi)=(q_{1}+V(x)+|\xi|^{2l})^{\frac{1}{2k_0}},
\qquad
M(x,\xi)=q_{1}+V(x)+|\xi|^{2l}.
\]
As discussed above, $\Phi$ and $\Psi$ are sublinear, temperate weights, and $M$ is temperate. Consequently, one can consider the symbol class $S(M;\Phi,\Psi)$, and we obtain the continuous embedding
\begin{align}
\label{Symbolinclusion}
S\bigl((q_{1}+V(x)+|\xi|^{2l})^{\frac{m}{2}},g\bigr)
\subset
S\bigl(M^{\frac{m}{2}};\Phi,\Psi\bigr).
\end{align}

Indeed, if $a\in S\bigl((q_{1}+V(x)+|\xi|^{2l})^{\frac{m}{2}},g\bigr)$, then by \eqref{symbolclassdef},
\begin{align*}
|\partial_x^{\beta}\partial_\xi^{\alpha} a(x,\xi)|
&\le
C_{\alpha,\beta}\bigl(q_{1}+V(x)^{1/2}+|\xi|^{l}\bigr)^{m-\frac{|\beta|}{k}-\frac{|\alpha|}{l}} \\
&\le
C'_{\alpha,\beta}(q_{1}+V(x)+|\xi|^{2l})^{\frac{m}{2}-\frac{|\beta|}{2k}-\frac{|\alpha|}{2l}} \\
&\le
C'_{\alpha,\beta}(q_{1}+V(x)+|\xi|^{2l})^{\frac{m}{2}-\frac{|\beta|}{2k_0}-\frac{|\alpha|}{2k_0}} \\
&=
C'_{\alpha,\beta}\,M^{\frac{m}{2}}\Psi^{-|\alpha|}\Phi^{-|\beta|},
\end{align*}
which proves \eqref{Symbolinclusion}.
The Weyl symbol of $\mathcal H_{k,l}$ is given by $V(x)+|\xi|^{2l}.$

For every multi-index $\alpha\in\mathbb{N}^d$, we have
\[
|\partial_\xi^{\alpha}(|\xi|^{2l})|
=
C_{\alpha}|\xi|^{2l-|\alpha|}
\lesssim
(q_{1}+V(x)^{1/2}+|\xi|^{l})^{2-\frac{|\alpha|}{l}},
\]
and similarly, for any $\beta\in\mathbb{N}^d$,
\[
|\partial_x^{\beta}V(x)|
\le
C_{\beta}|x|^{2k-|\beta|}
\lesssim
(q_{1}+V(x)^{1/2}+|\xi|^{l})^{2-\frac{|\beta|}{k}}.
\]
Using Lemma~5.4 of \cite{MR4489248}, it follows that,
\[
\mathcal{H}_{k,l}(x,\xi)\in\Sigma^{2}_{k,l},
\]
and therefore
\[
\mathcal{H}_{k,l}(x,\xi)\in S(M;\Phi,\Psi).
\]
Finally, there exist constants $C>0$ and $R>0$ such that
\[
|\xi|^{2l}+V(x)
\ge
C\,(q_{1}+V(x)+|\xi|^{2l}),
\qquad
|x|+|\xi|\ge R,
\]
which shows that $M$ is elliptic.

Therefore, by Remark~\ref{remark}, we conclude that
\[
\mathcal{H}_{k,l}(x,\xi)\in \mathrm{Hypo}(M,M_{0};\Phi,\Psi),
\]
where
\[
M(x,\xi)=M_{0}(x,\xi)=q_{1}+V(x)+|\xi|^{2l},
\qquad
\Phi(x,\xi)=\Psi(x,\xi)=(q_{1}+V(x)+|\xi|^{2l})^{\frac{1}{2k_{0}}},
\]
and $k_{0}=\max\{k,l\}.$
As a consequence, Theorem~4.2.9 in \cite{MR2668420} implies that the spectrum of the closure $\overline{\mathcal{H}_{k,l}}$ in $L^{2}(\mathbb{R}^{d})$ consists of a purely discrete sequence of eigenvalues diverging to $+\infty$.

It is important to note that the densely defined operator $\mathcal{H}_{k,l}$ on $L^{2}(\mathbb{R}^{d})$ is invertible. Indeed, since both $(-\Delta)^{l}$ and the potential $V(x)$ are strictly positive operators, all eigenvalues of $\mathcal{H}_{k,l}$ are strictly positive. In particular, $0$ does not belong to the spectrum of $\overline{\mathcal{H}_{k,l}}$.

Moreover, all eigenvalues of $\overline{\mathcal{H}_{k,l}}$ have finite multiplicity, and the corresponding eigenfunctions belong to the Schwartz space $\mathcal{S}(\mathbb{R}^{d})$. Consequently, $L^{2}(\mathbb{R}^{d})$ admits an orthonormal basis consisting of eigenfunctions of $\overline{\mathcal{H}_{k,l}}$.

Let $\{\lambda_j\}_{j\ge 0}$ denote the eigenvalues of $\mathcal{H}_{k,l}$, arranged in increasing order and counted with multiplicity, with $\lambda_0>0$ denoting the smallest eigenvalue. Owing to the polynomial growth of the potential $V(x)\sim |x|^{2k}$ and the quasi-homogeneous structure of the symbol $|\xi|^{2l}+V(x)$, the eigenvalue asymptotics established at the end of Section~5 in \cite{MR4299820} for the harmonic oscillator $-\Delta+|x|^{2}$ extend to the operator $\mathcal{H}_{k,l}$. In particular, one obtains the following asymptotic behavior:
\begin{equation}
\label{eigenvalueestimate}
\lambda_j \sim C_{k,l}\, j^{\frac{2kl}{d(k+l)}},
\qquad \text{as } j\to\infty,
\end{equation}
for some constant $C_{k,l}>0$ depending only on $k$ and $l$.

Let $\{\Phi_{j,i}\}_{i=1}^{d_j}$ be an orthonormal basis of the eigenspace $H_j$ corresponding to the eigenvalue $\lambda_j$, and denote by $d_j=\dim H_j$ its multiplicity. Writing $P_j$ for the orthogonal projection onto $H_j$, we have the spectral decomposition
\begin{align*}
\mathcal{H}_{k,l} f
=
\sum_{j=0}^{\infty} \lambda_j P_j f,
\qquad
P_j f
=
\sum_{i=1}^{d_j} \langle f,\Phi_{j,i}\rangle \Phi_{j,i},
\end{align*}
where $\langle\cdot,\cdot\rangle$ denotes the inner product in $L^{2}(\mathbb{R}^{d})$.

For any $\beta\in\mathbb{R}$, the fractional powers of $\mathcal{H}_{k,l}$ are then defined via the spectral theorem by
\begin{align*}
\mathcal{H}_{k,l}^{\beta} f
=
\sum_{j=0}^{\infty} \lambda_j^{\beta} P_j f.
\end{align*}
We now introduce a class of function spaces naturally associated with the generalised anharmonic oscillator, namely the generalised anharmonic Sobolev spaces.

Let $\mathcal{H}_{k,l}$ be the generalised anharmonic oscillator. For $s\in\mathbb{R}$ and $k,l\in\mathbb{N}$, the corresponding Sobolev space $Q^{s}_{k,l}$ is defined by
\begin{align*}
Q^{s}_{k,l}
:=
\Bigl\{u\in\mathcal{S}'(\mathbb{R}^{d}) : 
\bigl\|(\mathcal{H}_{k,l})^{\frac{s}{2}}u\bigr\|_{L^{2}(\mathbb{R}^{d})}<\infty
\Bigr\}.
\end{align*}
When $l=1$ and $V(x)=|x|^{2}$, the space $Q^{s}_{k,l}$ coincides with the classical Shubin–Sobolev space, also referred to as the Hermite–Sobolev space.

Using the spectral decomposition of $\mathcal{H}_{k,l}$, the space $Q^{s}_{k,l}$ admits the equivalent characterization
\begin{align*}
\|f\|^{2}_{Q^{s}_{k,l}}
:=
\bigl\|(\mathcal{H}_{k,l})^{\frac{s}{2}}f\bigr\|^{2}_{L^{2}(\mathbb{R}^{d})}
=
\sum_{j=0}^{\infty}
\lambda_{j}^{s}\,
\|P_{j}f\|^{2}_{L^{2}(\mathbb{R}^{d})},
\end{align*}
where $\{\lambda_j\}_{j\ge0}$ are the eigenvalues of $\mathcal{H}_{k,l}$ and $P_j$ denotes the orthogonal projection onto the corresponding eigenspace. In particular, $Q^{s}_{k,l}$ consists precisely of those tempered distributions whose spectral coefficients exhibit the weighted $\ell^{2}$-summability dictated by the eigenvalue growth of $\mathcal{H}_{k,l}$.

The space $Q^{s}_{k,l}$ is a Hilbert space when endowed with the sesquilinear form
\begin{align*}
\langle u,v\rangle_{Q^{s}_{k,l}}
:=
\bigl\langle (\mathcal{H}_{k,l})^{\frac{s}{2}}u,\,
(\mathcal{H}_{k,l})^{\frac{s}{2}}v
\bigr\rangle_{L^{2}(\mathbb{R}^{d})}.
\end{align*}
Moreover, the dual space of $Q^{s}_{k,l}$ can be identified with $Q^{-s}_{k,l}$ via the $L^{2}$ pairing.

The proof of the following lemma follows verbatim from that of Lemma~4.4.19 in \cite{MR4201879}, upon replacing the standard weight by the modulation weight $v_{s}$ and using the embedding and duality properties of the spaces $Q^{s}_{k,l}$. We therefore omit the proof.

\begin{lemma}
\label{anharmonicsobolevspace}
For every $s\in\mathbb{R}$, 
\[
\mathcal{M}^{2,2}_{s}=Q^{s}_{k,l}
\]
holds with equivalence of norms.
\end{lemma}

As a consequence of Lemma~\ref{anharmonicsobolevspace}, together with the inclusion relations of Lemma~\ref{Inclusionrelation} and Hölder’s inequality, we obtain that for every $0<p,q\le\infty$ and for $|s|$ sufficiently large,
\begin{align*}
Q^{s}_{k,l}
\hookrightarrow
\mathcal{M}^{p,p}
\hookrightarrow
\mathcal{M}^{\infty,\infty}
\hookrightarrow
Q^{-s}_{k,l}.
\end{align*}

Recall that $\mathcal{M}^{p,q}$ denotes the unweighted modulation space $\mathcal{M}^{p,q}_{s}$ with $s=0$. For nonzero weights, the following refined embedding result holds.

\begin{lemma}
For suitably large parameters $s_{0}>s>0$, the following chain of continuous embeddings is valid:
\begin{align}
\label{chain1}
Q^{s_{0}}_{k,l}
\hookrightarrow
\mathcal{M}^{p,q}_{s}
\hookrightarrow
\mathcal{M}^{\infty,\infty}_{s}
\hookrightarrow
Q^{-s_{0}}_{k,l}.
\end{align}
\end{lemma}
\begin{proof}
We first consider the case $p,q\ge 2$. By the inclusion relation stated in Lemma~\ref{Inclusionrelation}, for parameters $s_{0}>s>0$ one has the continuous embedding
\[
\mathcal{M}^{2,2}_{s_{0}}\hookrightarrow \mathcal{M}^{p,q}_{s}.
\]
Since $\mathcal{M}^{2,2}_{s_{0}}=Q^{s_{0}}_{k,l}$ by Lemma~\ref{anharmonicsobolevspace}, it follows that
\begin{align}
\label{chain}
Q^{s_{0}}_{k,l}
=\mathcal{M}^{2,2}_{s_{0}}
\hookrightarrow
\mathcal{M}^{p,q}_{s}
\hookrightarrow
\mathcal{M}^{\infty,\infty}_{s}
\hookrightarrow
\mathcal{M}^{\infty,\infty}
\hookrightarrow
Q^{-s_{0}}_{k,l}.
\end{align}
This establishes \eqref{chain1} in the case $p,q\ge 2$.

We now turn to the case $0<p,q\le 2$. Without loss of generality, we may assume $p\le q$. In view of the well-known embedding
\[
\mathcal{M}^{p,p}_{s}\hookrightarrow \mathcal{M}^{p,q}_{s}, \qquad q\ge p,
\]
it suffices to prove that
\[
\mathcal{M}^{2,2}_{s_{0}}\hookrightarrow \mathcal{M}^{p,p}_{s}.
\]
Let $f\in \mathcal{M}^{2,2}_{s_{0}}$. Then, using the definition of modulation spaces and Hölder’s inequality, we obtain
\begin{align*}
\|V_{g}f\, v_{s}\|_{L^{p,p}}
&=
\|V_{g}f\, v_{s_{0}}\, v_{s-s_{0}}\|_{L^{p,p}} \\
&\le
\|V_{g}f\, v_{s_{0}}\|_{L^{2,2}}
\,
\|v_{s-s_{0}}\|_{L^{r,r}},
\end{align*}
where the exponents satisfy $\frac{1}{p}=\frac{1}{2}+\frac{1}{r}$. Choosing $s_{0}$ sufficiently large ensures that $v_{s-s_{0}}\in L^{r,r}$, and hence the embedding $\mathcal{M}^{2,2}_{s_{0}}\hookrightarrow \mathcal{M}^{p,p}_{s}$ follows.

Combining the above arguments, we conclude that the chain of embeddings \eqref{chain} holds in all cases, which proves \eqref{chain1}.
\end{proof}

We now state two key results that form the backbone of the analysis in this section. The first establishes boundedness of 
$t$-quantised operators associated with symbols in the anharmonic classes introduced above. This result is a direct consequence of Theorem~3.1 in \cite{MR3636061}, and its proof follows the same line of argument as Theorem~3.2 in \cite{MR4944933}. We therefore state it without proof.

\begin{theorem}
\label{quantisationtheorem}
Let $m\in\mathbb{R}$, $a\in\Sigma^{ -m }_{k,l}$, $0<p,q\leq \infty$, and $t\in\mathbb{R}$. Then the $t$-quantised operator $a_t(x,D)$ extends to a bounded operator
\[
a_t(x,D):\mathcal{M}^{p,q}\longrightarrow \mathcal{M}^{p,q}_{m},
\]
with operator norm depending only on finitely many seminorms of $a$ in $\Sigma^{ -m }_{k,l}$.
\end{theorem}

The following result shows that fractional powers of the generalised anharmonic oscillator admit a pseudodifferential representation with a positive elliptic symbol adapted to the anharmonic geometry. Since the argument is analogous to the proof of  \cite[Proposition~2.3]{MR4313961} and \cite[Theorem 3.3]{MR4944933}, with only minor modifications in our setting, we omit the details.

\begin{theorem}
\label{Calderontheorem}
Let $\beta>0,$ $V(x)$ be a strictly positive homogeneous potential of polynomial growth of
order $2k$ and $k,l\ge 1$ be integers. Then the fractional anharmonic oscillator
\[
\mathcal{H}_{k,l}^{\beta}
:=
\bigl((-\Delta)^{l}+V(x)\bigr)^{\beta}
\]
is a pseudodifferential operator whose Weyl symbol $\Hd^{\beta}(x,\xi)$ is real-valued and belongs to the class $\Sigma^{2\beta}_{k,l}$. More precisely, the symbol admits the asymptotic representation
\begin{equation}
\label{assymtotic}
\mathcal{H}_{k,l}^{\beta}(x,\xi)
=
\bigl(|\xi|^{2l}+V(x)\bigr)^{\beta}
+
r(x,\xi),
\qquad
|\xi|^{l}+|x|^{k}\ge 1,
\end{equation}
where the remainder satisfies
\[
r \in \Sigma^{\,2\beta-\frac{k+l}{kl}}_{k,l}.
\]
\end{theorem}
As mentioned above, for $\beta>0$ the operator $\mathcal{H}_{k,l}^{\beta}$ is a pseudodifferential operator with a positive Weyl symbol in the class $S(M^{\beta};\Phi,\Psi)$, and its spectrum consists of a discrete sequence of eigenvalues diverging to $+\infty$. The corresponding eigenfunctions $\{\Phi_{j}\}_{j\ge0}$ form an orthonormal basis of $L^{2}(\mathbb{R}^{d})$.

For $t\ge0$, we define the semigroup generated by $\mathcal{H}_{k,l}^{\beta}$ via the spectral theorem by
\[
e^{-t\mathcal{H}_{k,l}^{\beta}}f
=
\sum_{j=0}^{\infty}e^{-t\lambda_{j}^{\beta}}P_{j}f.
\]
In view of Theorem~4.5.1 in \cite{MR2668420}, the operator $e^{-t\mathcal{H}_{k,l}^{\beta}}$ is itself a pseudodifferential operator whose Weyl symbol satisfies a family of uniform estimates. These estimates play a crucial role in the boundedness and smoothing results established in the next theorem.

We are now in a position to state the main result of this section, which establishes precise smoothing and boundedness properties for the fractional anharmonic heat semigroup on weighted modulation spaces.

\begin{theorem}
\label{BoundofH}
Let $\beta>0$, $0<p_{1},p_{2},q_{1},q_{2}\leq\infty$, $s_{1}>0$, and $s_{2}\in\mathbb{R}$. Define
\begin{align*}
\frac{1}{\widetilde{p}}
:=
\max\Bigl\{\frac{1}{p_{2}}-\frac{1}{p_{1}},\,0\Bigr\},
\qquad
\frac{1}{\widetilde{q}}
:=
\max\Bigl\{\frac{1}{q_{2}}-\frac{1}{q_{1}},\,0\Bigr\},
\end{align*}
and set
\begin{align}
\label{sigma}
\sigma
:=
\frac{d}{2\beta}
\left(
\frac{1}{k\,\widetilde{p}}
+
\frac{1}{l\,\widetilde{q}}
\right).
\end{align}
Then, for all $0<t\leq1$, the fractional anharmonic heat semigroup satisfies the estimate
\begin{equation}
\label{maininequality}
\bigl\|e^{-t\mathcal{H}_{k,l}^{\beta}}f\bigr\|_{\mathcal{M}^{p_{2},q_{2}}_{s_{2}}}
\le
C(t)\,
\|f\|_{\mathcal{M}^{p_{1},q_{1}}_{s_{1}}},
\end{equation}
where
\begin{equation}
\label{firstconstant}
C(t)=C_{0}\,t^{-\sigma}
\end{equation}
for some constant $C_{0}>0$ independent of $t$ and $f$.

In particular, when $s_{1}=s_{2}=s\ge0$, the estimate improves to
\begin{equation}
\label{secondconstant}
C(t)=
\begin{cases}
C_{0}\,t^{-\sigma}, & 0<t\le1,\\[0.3em]
C_{0}\,e^{-t\lambda_{0}^{\beta}}, & t\ge1,
\end{cases}
\end{equation}
where $\lambda_{0}>0$ denotes the smallest eigenvalue of the anharmonic oscillator $\mathcal{H}_{k,l}$.
\end{theorem}

\begin{proof}
By the embedding properties of modulation spaces recalled earlier, it suffices to establish \eqref{maininequality} with $p_{2}$ replaced by $\min\{p_{1},p_{2}\}$ and $q_{2}$ replaced by $\min\{q_{1},q_{2}\}$. Hence, without loss of generality, we may assume throughout the proof that
\[
p_{2}\le p_{1}
\qquad\text{and}\qquad
q_{2}\le q_{1}.
\]
Since the operator $\mathcal{H}_{k,l}^{\beta}$, $\beta>0$, is a pseudodifferential operator with a positive, elliptic Weyl symbol belonging to the class $S(M^{\beta};\Phi,\Psi)$, its $g$-ellipticity follows directly from the asymptotic expansion \eqref{assymtotic}. Consequently, Theorem~4.5.1 of \cite{MR2668420} applies to the associated heat semigroup $e^{-t\mathcal{H}_{k,l}^{\beta}}$.

In particular, the operator $e^{-t\mathcal{H}_{k,l}^{\beta}}$ is a pseudodifferential operator with Weyl symbol $b_{t}(x,\xi)$ depending on the parameter $t$. Moreover, for every $N\ge0$, the family $\{t^{N}b_{t}\}_{t\in[0,T]}$ is uniformly bounded in the symbol class $\Sigma^{ -2\beta N }_{k,l}$ for any fixed $T>0$.

By Theorem~\ref{quantisationtheorem}, it follows that
\begin{align*}
\|e^{-t\mathcal{H}_{k,l}^{\beta}}f\|_{\mathcal{M}^{p_{1},q_{1}}}
\lesssim
\|f\|_{\mathcal{M}^{p_{1},q_{1}}}.
\end{align*}
Since $s_{1}>0$ and $\mathcal{M}^{p_{1},q_{1}}_{s_{1}}\hookrightarrow \mathcal{M}^{p_{1},q_{1}}$, we obtain
\begin{align}
\label{modulationinclusion1}
\|V_{g}(e^{-t\mathcal{H}_{k,l}^{\beta}}f)\|_{L^{p_{1},q_{1}}}
\lesssim
\|f\|_{\mathcal{M}^{p_{1},q_{1}}}
\lesssim
\|f\|_{\mathcal{M}^{p_{1},q_{1}}_{s_{1}}}.
\end{align}

Next, since $t^{N}b_{t}\in\Sigma^{ -2\beta N }_{k,l}$ for any $N\in\mathbb{N}$, an argument analogous to the one above yields
\begin{align}
\label{modulationinclusion2}
\|t^{N}v_{2\beta N}V_{g}(e^{-t\mathcal{H}_{k,l}^{\beta}}f)\|_{L^{p_{1},q_{1}}}
\lesssim
\|f\|_{\mathcal{M}^{p_{1},q_{1}}_{s_{1}}}.
\end{align}
Combining \eqref{modulationinclusion1} and \eqref{modulationinclusion2}, we infer that
\begin{align}
\label{eqn012}
\|(1+t^{N}v_{2\beta N})V_{g}(e^{-t\mathcal{H}_{k,l}^{\beta}}f)\|_{L^{p_{1},q_{1}}}
\le
C\,
\|f\|_{\mathcal{M}^{p_{1},q_{1}}_{s_{1}}},
\end{align}
where the constant $C$ is independent of $t\in(0,1]$.

Applying Hölder’s inequality, we obtain
\begin{align}
\label{eqn1234}
\begin{aligned}
\|v_{s_{2}}V_{g}(e^{-t\mathcal{H}_{k,l}^{\beta}}f)\|_{L^{p_{2},q_{2}}}
&=
\|v_{s_{2}}(1+t^{N}v_{2\beta N})^{-1}
(1+t^{N}v_{2\beta N})
V_{g}(e^{-t\mathcal{H}_{k,l}^{\beta}}f)\|_{L^{p_{2},q_{2}}} \\
&\lesssim
\|v_{s_{2}}(1+t^{N}v_{2\beta N})^{-1}\|_{L^{\widetilde{p},\widetilde{q}}}
\,
\|(1+t^{N}v_{2\beta N})
V_{g}(e^{-t\mathcal{H}_{k,l}^{\beta}}f)\|_{L^{p_{1},q_{1}}},
\end{aligned}
\end{align}
where the indices satisfy
\[
\frac{1}{p_{2}}=\frac{1}{p_{1}}+\frac{1}{\widetilde{p}},
\qquad
\frac{1}{q_{2}}=\frac{1}{q_{1}}+\frac{1}{\widetilde{q}}.
\]

Using \eqref{eqn012} in \eqref{eqn1234}, we arrive at
\begin{align}
\label{beforeenequality}
\|e^{-t\mathcal{H}_{k,l}^{\beta}}f\|_{\mathcal{M}^{p_{2},q_{2}}_{s_{2}}}
\lesssim
\|v_{s_{2}}(1+t^{N}v_{2\beta N})^{-1}\|_{L^{\widetilde{p},\widetilde{q}}}
\,
\|f\|_{\mathcal{M}^{p_{1},q_{1}}_{s_{1}}},
\end{align}
which reduces the proof to estimating the weighted $L^{\widetilde{p},\widetilde{q}}$ norm on the right-hand side.
We now estimate the quantity
\[
C(t):=\bigl\|v_{s_{2}}(1+t^{N}v_{2\beta N})^{-1}\bigr\|_{L^{\widetilde{p},\widetilde{q}}}.
\]
By definition of the weight $v$, we have
\begin{align*}
C(t)
&=
\left\|
\frac{(q_{1}+V(x)^{1/2}+|\xi|^{l})^{s_{2}}}
{1+t^{N}(q_{1}+V(x)^{1/2}+|\xi|^{l})^{2\beta N}}
\right\|_{L^{\widetilde{p},\widetilde{q}}} \\
&\le\left\| \frac{(q_{1}+V(x)^{1/2}+|\xi|^{l})^{s_{2}}}{1+t^{N}(V(x)^{1/2}+|\xi|^{l})^{2\beta N}}          \right\|_{L^{\widetilde{p},\widetilde{q}}}
\end{align*}
Since the potential satisfies $V(x)\sim |x|^{2k}$ for large $|x|$, we obtain
\begin{align*}
C(t)^{\widetilde{q}}
&\lesssim
\left\| \frac{(q_{1}+|x|^{k}+|\xi|^{l})^{s_{2}}}{1+t^{N}(|x|^{k}+|\xi|^{l})^{2\beta N}}          \right\|^{\widetilde{q}}_{L^{\widetilde{p},\widetilde{q}}}\lesssim \left\| \frac{(1+|x|^{k}+|\xi|^{l})^{s_{2}}}{1+t^{N}(|x|^{k}+|\xi|^{l})^{2\beta N}}          \right\|^{\widetilde{q}}_{L^{\widetilde{p},\widetilde{q}}}.
\end{align*}
Using convexity and the fact that same degree polynomials have equivalent growth  at infinity  we obtain
\begin{align*}
C(t)^{\widetilde{q}}&\lesssim \left\| \frac{(1+|x|^{k}+|\xi|^{l})^{s_{2}}}{(1+t^{\frac{1}{2\beta}}(|x|^{k}+|\xi|^{l}))^{2\beta N}}          \right\|^{\widetilde{q}}_{L^{\widetilde{p},\widetilde{q}}}\\
&\lesssim \| (1+t^{\frac{1}{2\beta}}(|x|^{k}+|\xi|^{l}))^{s_{2}-2\beta N}\|^{\widetilde{q}}_{L^{\widetilde{p},\widetilde{q}}}\\
&=\int_{\mathbb{R}^{d}}\left(\int_{\mathbb{R}^{d}} [(1+t^{\frac{1}{2\beta}}(|x|^{k}+|\xi|^{l}))^{2\beta N-s_{2}}]^{-\widetilde{p}}\ \mathrm{d}x\right)^{\frac{\widetilde{q}}{\widetilde{p}}}\ \mathrm{d}\xi
\end{align*}

Introducing the scaling
\[
x \mapsto t^{\frac{1}{2\beta k}}\,x,
\qquad
\xi \mapsto t^{\frac{1}{2\beta l }}\,\xi,
\]
we can rewrite the above integral as
\begin{align*}
C(t)^{\widetilde{q}}
&\lesssim
\int_{\mathbb{R}^{d}}
\left(
\int_{\mathbb{R}^{d}}
\bigl( 1+
|t^{\frac{1}{2\beta k}}x|^{k}
+
|t^{\frac{1}{2\beta l}}\xi|^{l}
\bigr)^{-(2\beta N-s_{2})\widetilde{p}}
\,\mathrm{d}x
\right)^{\frac{\widetilde{q}}{\widetilde{p}}}
\mathrm{d}\xi .
\end{align*}

Next making the change of variables
\[
\widetilde{x}
=
t^{\frac{1}{2\beta k}}\,x,
\qquad
\widetilde{\xi}
=
t^{\frac{1}{2\beta l}}\,\xi,
\]
we obtain
\begin{align*}
C(t)^{\widetilde{q}}
&\lesssim
\int_{\mathbb{R}^{d}}
\left(
\int_{\mathbb{R}^{d}}
\bigl(1+|\widetilde{x}|^{k}+|\widetilde{\xi}|^{l}\bigr)^{-(2\beta N-s_{2})\widetilde{p}}
\,t^{-\frac{d}{2\beta k}}
\,\mathrm{d}\widetilde{x}
\right)^{\frac{\widetilde{q}}{\widetilde{p}}}
t^{-\frac{d}{2\beta l}}
\,\mathrm{d}\widetilde{\xi} \\
&=
t^{-\frac{d}{2\beta}(\frac{\widetilde{q}}{k\widetilde{p}}+\frac{1}{l})}
\int_{\mathbb{R}^{d}}
\left(
\int_{\mathbb{R}^{d}}
\bigl(1+|\widetilde{x}|^{k}+|\widetilde{\xi}|^{l}\bigr)^{-(2\beta N-s_{2})\widetilde{p}}
\,\mathrm{d}\widetilde{x}
\right)^{\frac{\widetilde{q}}{\widetilde{p}}}
\mathrm{d}\widetilde{\xi}.
\end{align*}

Denoting
\[
I^{\widetilde{q}}
:=
\int_{\mathbb{R}^{d}}
\left(
\int_{\mathbb{R}^{d}}
\bigl(1+|\widetilde{x}|^{k}+|\widetilde{\xi}|^{l}\bigr)^{-(2\beta N-s_{2})\widetilde{p}}
\,\mathrm{d}\widetilde{x}
\right)^{\frac{\widetilde{q}}{\widetilde{p}}}
\mathrm{d}\widetilde{\xi},
\]
we note that \(I<\infty\) provided \(N\) is chosen sufficiently large. With this choice, the preceding estimate yields
\begin{align}
\label{snot}
C(t)^{\widetilde{q}}
\nonumber&\lesssim
I^{\widetilde{q}}\,
t^{-\frac{d}{2\beta }
\left(\frac{\widetilde{q}}{k\widetilde{p}}+\frac{1}{l}\right)} \\
&\le
C'\,
t^{-\frac{d}{2\beta}
\left(\frac{1}{k\widetilde{p}}+\frac{1}{l\widetilde{q}}\right)},
\end{align}
for some constant $C'>0$ independent of $t\in(0,1]$. Consequently,
\[
C(t)
\lesssim
t^{-\sigma},
\qquad
0<t\le1,
\]
where
\[
\sigma
=
\frac{d}{2\beta}
\left(
\frac{1}{k\widetilde{p}}+\frac{1}{l\widetilde{q}}
\right).
\]

Inserting this bound into \eqref{beforeenequality}, we finally obtain
\begin{align}
\label{so}
\|e^{-t\mathcal{H}_{k,l}^{\beta}}f\|_{\mathcal{M}^{p_{2},q_{2}}_{s_{2}}}
\le
C_{0}\,t^{-\sigma}
\|f\|_{\mathcal{M}^{p_{1},q_{1}}_{s_{1}}}.
\end{align}

We now consider the special case $s_{1}=s_{2}=s\ge0$. The proof is divided into two regimes, namely $t\ge1$ and $0<t\le1$.

\medskip
\noindent
\textbf{Case 1: $0<t\leq 1$.}
In this case, for $s_{1}=s_{2}=s\ge 0$, the estimate \eqref{so} holds. The proof follows the same argument as in the case $s_{1}\neq s_{2}$.

\textbf{Case 2: $t\ge1$.}
Using the spectral decomposition of $\mathcal{H}_{k,l}$, we write
\begin{align}
\label{semigroupPk}
\begin{aligned}
\|e^{-t\mathcal{H}_{k,l}^{\beta}}f\|_{\mathcal{M}^{p_{2},q_{2}}_{s}}
&=
\Bigl\|\sum_{j=0}^{\infty}e^{-t\lambda_{j}^{\beta}}P_{j}f\Bigr\|_{\mathcal{M}^{p_{2},q_{2}}_{s}} \\
&\le
\sum_{j=0}^{\infty}e^{-t\lambda_{j}^{\beta}}
\|P_{j}f\|_{\mathcal{M}^{p_{2},q_{2}}_{s}} \\
&\le
\sum_{j=0}^{\infty}e^{-t\lambda_{j}^{\beta}}
\|P_{j}\|_{\mathcal{M}^{p_{1},q_{1}}_{s}\to\mathcal{M}^{p_{2},q_{2}}_{s}}
\|f\|_{\mathcal{M}^{p_{1},q_{1}}_{s}} .
\end{aligned}
\end{align}

Let $s_{0}>s$ be sufficiently large. By the embedding relations established earlier,
\[
Q^{s_{0}}_{k,l}
\hookrightarrow
\mathcal{M}^{p,q}_{s}
\hookrightarrow
\mathcal{M}^{\infty,\infty}_{s}
\hookrightarrow
Q^{-s_{0}}_{k,l}.
\]
Hence there exists a constant $C'>0$ such that
\begin{align}
\label{Qsinequality1}
\|P_{j}\|_{\mathcal{M}^{p_{1},q_{1}}_{s}\to\mathcal{M}^{p_{2},q_{2}}_{s}}
\le
C'
\sup_{\|f\|_{Q^{-s_{0}}_{k,l}}\le1}
\|P_{j}f\|_{Q^{s_{0}}_{k,l}} .
\end{align}

By orthogonality of the spectral projections,
\[
\|P_{j}f\|_{Q^{s_{0}}_{k,l}}^{2}
=
\sum_{i=0}^{\infty}\lambda_{i}^{s_{0}}\|P_{i}P_{j}f\|_{L^{2}}^{2}
=
\lambda_{j}^{s_{0}}\|P_{j}f\|_{L^{2}}^{2},
\]
and therefore
\begin{equation}
\label{Qsinequality2}
\|P_{j}f\|_{Q^{s_{0}}_{k,l}}
=
\lambda_{j}^{\frac{s_{0}}{2}}\|P_{j}f\|_{L^{2}}.
\end{equation}
On the other hand, if $\|f\|_{Q^{-s_{0}}_{k,l}}\le1$, then
\[
\sum_{j=0}^{\infty}\lambda_{j}^{-s_{0}}\|P_{j}f\|_{L^{2}}^{2}\le1,
\]
which implies
\begin{equation}
\label{Qsinequality3}
\|P_{j}f\|_{L^{2}}\le \lambda_{j}^{\frac{s_{0}}{2}}.
\end{equation}
Combining \eqref{Qsinequality2} and \eqref{Qsinequality3} in \eqref{Qsinequality1}, we obtain
\[
\|P_{j}\|_{\mathcal{M}^{p_{1},q_{1}}_{s}\to\mathcal{M}^{p_{2},q_{2}}_{s}}
\le
C'\lambda_{j}^{s_{0}}.
\]

Inserting this bound into \eqref{semigroupPk} yields
\[
\|e^{-t\mathcal{H}_{k,l}^{\beta}}\|_{\mathcal{M}^{p_{1},q_{1}}_{s}\to\mathcal{M}^{p_{2},q_{2}}_{s}}
\le
C'
\sum_{j=0}^{\infty}e^{-t\lambda_{j}^{\beta}}\lambda_{j}^{s_{0}}.
\]
Now following the same argument as in \cite{MR4944933} in Theorem~3.4 we obtain,
\begin{equation}
\label{eestimate}
\sum_{j=0}^{\infty}e^{-t\lambda_{j}^{\beta}}\lambda_{j}^{s_{0}}
\le
C_{0}e^{-t\lambda_{0}^{\beta}},
\qquad t\ge1,
\end{equation}
for some constant $C_{0}>0$.

This completes the proof in the case $t\ge1$.

\end{proof}

\section{Application to Well-posedness of the Nonlinear Heat Equations} 
As an application of the results obtained in the previous section, we study the well-posedness of nonlinear heat equations associated with fractional powers of the generalised anharmonic oscillator $\Hd^{\beta}$, involving different inhomogeneous nonlinearities. In particular, we consider the nonlinear heat equation with a power-type inhomogeneous nonlinearity
\begin{align}
\label{Nonlinearequation}
\begin{cases}
\partial_{t}u(t,x) + \Hd^{\beta} u(t,x) = \lambda (|u|^{2\nu}u)(t,x),\\
u(0,x) = u_{0}(x),
\end{cases}
\end{align}
for $(t,x)\in[0,\infty)\times\mathbb{R}^{d}$, where $\nu\in\mathbb{N}$, $\lambda\in\mathbb{C}$, and $\beta>0$.

We also investigate a nonlinear heat equation with a spatially inhomogeneous power nonlinearity
\begin{align}
\label{inhomogeneousequationforHd}
\begin{cases}
\partial_{t}u(t,x) + \Hd^{\beta}u(t,x)
= \lambda |x|^{-\alpha}\big(|u|^{2\nu}u\big)(t,x),\\
u(0,x) = u_{0}(x),
\end{cases}
\end{align}
for $(t,x)\in[0,\infty)\times\mathbb{R}^{d}$, where $u(t,x)\in\mathbb{C}$, $\lambda=\pm1$, $\alpha>0$, and $\nu\in\mathbb{N}$.
The well-posedness of \eqref{Nonlinearequation} has so far been established only in the setting of unweighted modulation spaces $\mathcal{M}^{p,q}$ with $p,q\geq 1$; see \cite{MR4313961, MR4944933}. In contrast, a conjecture of Cordero \cite{MR4215324} predicts that this class of nonlinear evolution equations should remain well posed in the substantially larger scale of weighted modulation spaces $\mathcal{M}^{p,q}_{s}$, including the quasi-Banach regime $0<q\leq 1$.

Relying on the sharp semigroup bounds for $e^{-t\Hd^{\beta}}$ established above, we confirm this prediction in the Banach range and prove that the Cauchy problem \eqref{Nonlinearequation} admits a unique global-in-time solution in $\mathcal{M}^{p,q}_{s}$ for all $1\leq p,q<\infty$ and for a suitable range of weights $s>0$. This result substantially extends the existing theory and demonstrates that weighted modulation spaces provide a natural and robust phase-space framework for nonlinear diffusion equations associated with anharmonic operators. We state the following theorem now. 
\begin{theorem} 
\label{firstwellposed}
Let $1 \leq p,q \leq \infty$ and assume that $s > \frac{d}{q'}$. There exists $\varepsilon>0$ such that, for every initial data
$u_{0} \in \mathcal{M}^{p,q}_{s}$ satisfying
\[
\|u_{0}\|_{\mathcal{M}^{p,q}_{s}} \leq \varepsilon,
\]
the Cauchy problem \eqref{Nonlinearequation} admits a unique global-in-time solution
\[
u \in L^{\infty}\big([0,\infty), \mathcal{M}^{p,q}_{s}\big).
\]
Moreover, if $1 \leq p,q < \infty$, then the solution depends continuously on time and satisfies
\[
u \in C\big([0,\infty), \mathcal{M}^{p,q}_{s}\big).
\]
\end{theorem}
\begin{proof}
  By Duhamel’s principle, the solution $u$ of \eqref{Nonlinearequation} can be written in the integral form
\begin{align}\label{eq4.5}
u(t)
= e^{-t\Hd^\beta}u_0
+ \lambda \int_{0}^{t} e^{-(t-s)\Hd^\beta}
\big(|u(s)|^{2\nu}u(s)\big)\,\mathrm{d}s
=: \mathcal{L}(u).
\end{align}

From Theorem~\ref{BoundofH}, for $p,q$ as in the statement of the theorem, the semigroup $e^{-t\Hd^\beta}$ satisfies the estimate
\begin{align}\label{eq4.7}
\| e^{-t\Hd^\beta} f \|_{\mathcal{M}_{s}^{p,q}}
\leq C'
\begin{cases}
t^{-\sigma}\, \| f \|_{\mathcal{M}_{s}^{p,q}}, & 0 \leq t \leq 1,\\[4pt]
e^{-t\lambda_0^\beta}\, \| f \|_{\mathcal{M}_{s}^{p,q}}, & t \geq 1,
\end{cases}
\end{align}
for some constant $C'>0$. In the present setting, $\sigma = 0$. Moreover, for $t \geq 1$, the exponential decay yields the uniform bound
\[
C' e^{-t\lambda_0^\beta} \leq C'.
\] 
Therefore, we obtain
\begin{align}\label{ocd}
\big\| e^{-t\Hd^\beta}u_0\big\|_{\mathcal{M}_{s}^{p,q}}
\leq C_{2}\,\|u_{0}\|_{\mathcal{M}^{p,q}_{s}},
\end{align}
for some constant $C_{2}>0$.  

Next, we estimate the inhomogeneous term in \eqref{eq4.5}. Applying Minkowski’s integral inequality together with the semigroup bound \eqref{eq4.7}, we obtain
\begin{align}\label{eq48}
\Bigg\| \int_{0}^{t} e^{-(t-s)\Hd^{\beta}}
\big(|u(s)|^{2\nu} u(s)\big)\,\mathrm{d}s \Bigg\|_{\mathcal{M}_{s}^{p,q}}
&\leq  \int_{0}^{t}
\big\| e^{-(t-s)\Hd^{\beta}}
\big(|u(s)|^{2\nu} u(s)\big) \big\|_{\mathcal{M}_{s}^{p,q}}\,\mathrm{d}s \nonumber\\
&\leq C'
\begin{cases}
\displaystyle \int_{0}^{t}
(t-s)^{-\sigma}
\big\| |u(s)|^{2\nu} u(s) \big\|_{\mathcal{M}_{s}^{p,q}}
\,\mathrm{d}s, & 0<t\leq 1,\\[10pt]
\displaystyle \int_{0}^{t}
e^{-(t-s)\lambda_0^\beta}
\big\| |u(s)|^{2\nu} u(s) \big\|_{\mathcal{M}_{s}^{p,q}}
\,\mathrm{d}s, & t \geq 1.
\end{cases}
\end{align}
Since $s-\frac{d}{q'}>0$, it follows from Theorem~\ref{Inclusionrelation} that
$\mathcal{M}^{p,q}_{s}$ is a multiplicative algebra. Consequently, we obtain
\begin{align}\label{ocd1}
\Bigg\| \int_{0}^{t} e^{-(t-s)\Hd^{\beta}}
\big(|u(s)|^{2\nu} u(s)\big)\,\mathrm{d}s \Bigg\|_{\mathcal{M}_{s}^{p,q}}
&\leq C'
\begin{cases}
\displaystyle \int_{0}^{t}
(t-s)^{-\sigma}
\|u(s)\|_{\mathcal{M}_{s}^{p,q}}^{2\nu+1}\,\mathrm{d}s,
& 0<t\leq 1,\\[10pt]
\displaystyle \int_{0}^{t}
e^{-(t-s)\lambda_0^\beta}
\|u(s)\|_{\mathcal{M}_{s}^{p,q}}^{2\nu+1}\,\mathrm{d}s,
& t \geq 1,
\end{cases}
\nonumber\\
&\leq C''\,\|u\|_{L^{\infty}([0,t],\mathcal{M}^{p,q}_{s})}^{2\nu+1}
\begin{cases}
\displaystyle \int_{0}^{t}\tau^{-\sigma}\,\mathrm{d}\tau, & 0<t\leq 1,\\[8pt]
\displaystyle \int_{0}^{t} e^{-\tau\lambda_{0}^{\beta}}\,\mathrm{d}\tau,
& t \geq 1,
\end{cases}
\nonumber\\
&\leq C_{3}\,
\|u\|_{L^{\infty}([0,\infty),\mathcal{M}^{p,q}_{s})}^{2\nu+1},
\end{align}
where we used that the integrals in the previous line are finite under the assumption
$\sigma=0$ for $0<t\leq 1$.

Combining \eqref{ocd} with \eqref{ocd1}  , we conclude that
\begin{align}
\label{C4eq}
\|\mathcal{L}(u)\|_{L^{\infty}([0,\infty),\mathcal{M}^{p,q}_{s})}
\leq C_{4}
\left(
\|u_{0}\|_{\mathcal{M}^{p,q}_{s}}
+
\|u\|_{L^{\infty}([0,\infty),\mathcal{M}^{p,q}_{s})}^{2\nu+1}
\right).
\end{align}
for some $C_{4}>0$.
 For $\varepsilon>0$, define
\[
B_{\varepsilon}
=
\Big\{
u \in L^{\infty}\big([0,\infty), \mathcal{M}^{p,q}_{s}\big)
:\ 
\|u\|_{L^{\infty}([0,\infty), \mathcal{M}^{p,q}_{s})}
\leq \varepsilon
\Big\},
\]
which is a closed ball centered at the origin in
$L^{\infty}\big([0,\infty), \mathcal{M}^{p,q}_{s}\big)$.

We first show that, for a suitable choice of $\varepsilon>0$, the mapping
$\mathcal{L}$ maps $B_{\varepsilon}$ into itself.  
Assume that the initial data satisfies
\[
\|u_{0}\|_{\mathcal{M}^{p,q}_{s}} \leq \frac{\varepsilon}{2C_{4}}.
\]
Then, by \eqref{C4eq}, for any $u \in B_{\varepsilon}$ we obtain
\[
\|\mathcal{L}(u)\|_{L^{\infty}([0,\infty), \mathcal{M}^{p,q}_{s})}
\leq \frac{\varepsilon}{2} + C_{4}\varepsilon^{2\nu+1}.
\]

Since $\nu>0$, we may choose $\varepsilon>0$ sufficiently small so that
\[
\varepsilon^{2\nu} \leq \frac{1}{2C_{4}}.
\]
With this choice, it follows that
\[
\|\mathcal{L}(u)\|_{L^{\infty}([0,\infty), \mathcal{M}^{p,q}_{s})}
\leq \frac{\varepsilon}{2} + \frac{\varepsilon}{2}
= \varepsilon,
\]
that is, $\mathcal{L}(u) \in B_{\varepsilon}$.

We next show that $\mathcal{L}$ is a contraction on $B_{\varepsilon}$.  
Let $u,v \in B_{\varepsilon}$. Repeating the arguments above and using the pointwise inequality
\[
\big||u|^{2\nu}u - |v|^{2\nu}v\big|
\leq C_{5}\big(|u|^{2\nu-1} + |v|^{2\nu-1}\big)|u-v|,
\]
for some constant $C_{5}>0$, we obtain
\begin{align*}
\|\mathcal{L}(u)-\mathcal{L}(v)\|_{L^{\infty}([0,\infty), \mathcal{M}^{p,q}_{s})}
&\leq
C_{5}\Big(
\|u\|^{2\nu-1}_{L^{\infty}([0,\infty), \mathcal{M}^{p,q}_{s})}
+
\|v\|^{2\nu-1}_{L^{\infty}([0,\infty), \mathcal{M}^{p,q}_{s})}
\Big)
\|u-v\|_{L^{\infty}([0,\infty), \mathcal{M}^{p,q}_{s})} \\
&\leq
C_{5}'\,\varepsilon^{2\nu-1}
\|u-v\|_{L^{\infty}([0,\infty), \mathcal{M}^{p,q}_{s})},
\end{align*}
for some constant $C_{5}'>0$.
Hence, by choosing $\varepsilon>0$ sufficiently small, for instance
\[
\varepsilon < \left(\frac{1}{2C_{5}'}\right)^{\frac{1}{2\nu-1}},
\]
we conclude that $\mathcal{L}$ is a contraction on $B_{\varepsilon}$.  
Therefore, by Banach’s contraction principle, the mapping $\mathcal{L}$ admits a unique fixed point in $B_{\varepsilon}$.  
This fixed point is the unique global solution of the Cauchy problem \eqref{Nonlinearequation}.

To establish that the unique solution belongs to
$C\big([0,\infty), \mathcal{M}^{p,q}_{s}\big)$ for $1 \leq p,q < \infty$,
we first assume that the semigroup $e^{-t\Hd^{\beta}}$ is strongly
continuous on $\mathcal{M}^{p,q}_{s}$. Under this assumption, the Banach
fixed point argument developed above can be carried out verbatim with
the space $C\big([0,\infty), \mathcal{M}^{p,q}_{s}\big)$ in place of
$L^{\infty}\big([0,\infty), \mathcal{M}^{p,q}_{s}\big)$, yielding a
solution continuous in time.

It therefore remains to justify the strong continuity of the semigroup
$e^{-t\Hd^{\beta}}$ on $\mathcal{M}^{p,q}_{s}$. In view of the estimate
\eqref{ocd}, it suffices to verify that, for every $f$ belonging to a
dense subset of $\mathcal{M}^{p,q}_{s}$, the mapping
\[
t \longmapsto e^{-t\Hd^{\beta}} f
\]
is continuous with values in $\mathcal{M}^{p,q}_{s}$.

Since the Schwartz class $\mathcal{S}(\mathbb{R}^{d})$ is dense in
$\mathcal{M}^{p,q}_{s}$, we fix $f \in \mathcal{S}(\mathbb{R}^{d})$.
It is well known that the semigroup $e^{-t\Hd^{\beta}}$ is strongly
continuous on $L^{2}(\mathbb{R}^{d})$, and that $\Hd^{\gamma}$ commutes
with $e^{-t\Hd^{\beta}}$ for every $\gamma \in \mathbb{N}$. Consequently,
for each $\gamma \in \mathbb{N}$, the mapping
\[
t \longmapsto e^{-t\Hd^{\beta}}\Hd^{\gamma}f
= \Hd^{\gamma} e^{-t\Hd^{\beta}}f
\]
is continuous with values in $L^{2}(\mathbb{R}^{d})$.

By a standard abstract argument (see \cite[p.~194]{MR2668420}), the
family of seminorms
\[
p_{\gamma}(f) := \|\Hd^{\gamma}f\|_{L^{2}(\mathbb{R}^{d})},
\qquad \gamma \in \mathbb{N},
\]
forms an equivalent system of seminorms on $\mathcal{S}(\mathbb{R}^{d})$.
Therefore, for any $t_{0} \in [0,\infty)$, convergence
\[
e^{-t\Hd^{\beta}}f \longrightarrow e^{-t_{0}\Hd^{\beta}}f
\quad \text{in } \mathcal{S}(\mathbb{R}^{d}) \text{ as } t \to t_{0}
\]
is equivalent to
\[
\|\Hd^{\gamma}(e^{-t\Hd^{\beta}}f)
- \Hd^{\gamma}(e^{-t_{0}\Hd^{\beta}}f)\|_{L^{2}(\mathbb{R}^{d})}
\longrightarrow 0
\quad \text{as } t \to t_{0},
\]
or, equivalently,
\[
\|e^{-t\Hd^{\beta}}\Hd^{\gamma}f
- e^{-t_{0}\Hd^{\beta}}\Hd^{\gamma}f\|_{L^{2}(\mathbb{R}^{d})}
\longrightarrow 0
\quad \text{as } t \to t_{0}.
\]

This shows that the mapping
\[
t \longmapsto e^{-t\Hd^{\beta}}f
\]
is continuous with values in $\mathcal{S}(\mathbb{R}^{d})$, and hence,
by continuity of the embedding, also when regarded as a
$\mathcal{M}^{p,q}_{s}$-valued function. Consequently, for
$1 \leq p,q < \infty$, the unique solution $u$ of
\eqref{Nonlinearequation} satisfies
\[
u \in C\big([0,\infty), \mathcal{M}^{p,q}_{s}\big).
\]
\end{proof}

\begin{remark} By adapting the method developed in \cite{MR4313961, MR4944933}, for usual modulation spaces, one can also establish the above well-posedness result, Theorem~\ref{firstwellposed}, in the weighted modulation space $\mathcal{M}^{p,q}_{s}$. However, this strategy imposes certain restrictions on the parameters $\beta$ and $\nu$. To avoid these constraints, we  use the fact that, for the chosen range of $s$, the space $\mathcal{M}^{p,q}_{s}$ is a multiplicative algebra. This structural property allows us to treat the equation in a wider class of modulation spaces without further restrictions on the parameters.

 Now for $s=0$ and $q=1$, using the approach in Theorem~\ref{firstwellposed} in this paper, one obtains a well-posedness result in the modulation space  $\mathcal{M}^{p,1}$. For $q=1$, well-posedness for the  classical Laplacian  was established in \cite{MR2506839}. In contrast, when $q \neq 1$, one may instead employ the method developed in \cite{MR4313961, MR4944933}, which naturally imposes certain restrictions on the parameters $\beta$ and $\nu$. Despite these constraints, this approach is particularly effective for the treatment of nonlinear heat equation with a spatially inhomogeneous power nonlinearity, as demonstrated in Theorem~\ref{Anharinhom}.

\end{remark}

We now consider the inhomogeneous nonlinear  equation~\eqref{inhomogeneousequationforHd}, which can also be treated within the framework of weighted
modulation spaces $\mathcal{M}^{p,q}_{s}$, assuming small initial data
$u_{0}$. To the best of our knowledge, the well-posedness theory for
inhomogeneous nonlinear heat equations of the form
\eqref{inhomogeneousequationforHd} associated with fractional powers of
the generalised anharmonic oscillator has not been addressed in the
existing literature.

The main difficulty in the analysis of \eqref{inhomogeneousequationforHd}
stems from the singularity of the nonlinear term. To address this
issue, we identify a range of exponents $\alpha>0$ for which the
singular weight
\[
f_{\alpha}(x) := |x|^{-\alpha}
\]
belongs to a suitable weighted modulation space
$\mathcal{M}^{p,q}_{s}$ for appropriate choices of $p,q$ and $s>0$.
This allows the nonlinear term to be handled within the modulation
space framework.

\begin{lemma}
\label{|x|inMpq}
Let $s>0$ and $ks<\alpha<d-ls$. Then $f_{\alpha} \in \mathcal{M}^{p,q}_{s}$
provided that
\[
p > \frac{d}{\alpha-ks}
\qquad \text{and} \qquad
q > \frac{d}{\,d-\alpha-ls\,}.
\]
\end{lemma}

\begin{proof} We decompose the function $f_{\alpha}$ as
\[
f_{\alpha} = \chi f_{\alpha} + (1-\chi)f_{\alpha},
\]
where $\chi \in C_c^{\infty}(\mathbb{R}^{d})$ is a smooth cutoff function
such that $\chi \equiv 1$ in a neighbourhood of the origin. We choose a
window function $g \in \mathcal{S}(\mathbb{R}^{d})$ with compact support.

For the local part $\chi f_{\alpha}$, the short-time Fourier transform
satisfies $V_{g}(\chi f_{\alpha})(x,\xi)=0$ for $|x|$ sufficiently large.
Recalling the argument in Example~2.1 in \cite{MR4313961}, for $x$ in a compact set, we have the following  pointwise estimate
\[
|V_{g}(\chi f_{\alpha})(x,\xi)|
\lesssim
|\widehat{f_{\alpha}}| *_{\xi} |\widehat{\chi}|
*_{\xi} |\widehat{\overline{g}}|(\xi).
\]

Note that $\chi f_{\alpha} \in \mathcal{M}^{p,q}_{s}$ if and only if
\[
\|V_{g}(\chi f_{\alpha})(x,\xi)(q_{1}+V(x)^{1/2}+|\xi|^{l})^{s}\|_{L^{p}_{x}L^{q}_{\xi}} < \infty,
\]
and observing that $V_{g}(\chi f_{\alpha})(\cdot,\xi)$ has compact
support in the $x$-variable, we obtain
\begin{align}
\label{eqns}
\|V_{g}(\chi f_{\alpha})(q_{1}+V(x)^{1/2}+|\xi|^{l})^{s}\|_{L^{p}_{x}L^{q}_{\xi}}
\nonumber&=
\Big\|
\|V_{g}(\chi f_{\alpha})(q_{1}+V(\cdot)^{1/2}+|\xi|^{l})^{s}\|_{L^{p}_{x}}
\Big\|_{L^{q}_{\xi}} \\
&\lesssim
K\,
\big\|
(1+|\xi|)^{ls}
\big(
|\widehat{f_{\alpha}}| *_{\xi} |\widehat{\chi}|
*_{\xi} |\widehat{\overline{g}}|
\big)(\xi)
\big\|_{L^{q}_{\xi}},
\end{align}
where the constant $K>0$ depends only on the measure of the compact support of $V_{g}(\chi f_{\alpha})(\cdot,\xi)$. To obtain the inequality \eqref{eqns} we use the fact that in the compact support in $x$ variable 
\begin{align*}
q_{1}+V(x)^{1/2}+|\xi|^{l}\leq C+|\xi|^{l}\lesssim (1+\xi)^{l},
\end{align*}
where $C>0$ is a constant depends on the compact support.
It is well known that $\widehat{f_{\alpha}} = C_{\alpha} f_{d-\alpha}$ for
some constant $C_{\alpha} \in \mathbb{R}$. Since both $|\widehat{\chi}|$
and $|\widehat{g}|$ are Schwartz functions, their convolution
$|\widehat{\chi}| * |\widehat{g}|$ also belongs to the Schwartz class.
Convolution with a Schwartz function regularizes singularities but does
not improve the decay at infinity; therefore,
\[
|\widehat{f_{\alpha}}| *_{\xi} |\widehat{\chi}|
*_{\xi} |\widehat{\overline{g}}|(\xi)
\sim |\xi|^{-(d-\alpha)}
\qquad \text{as } |\xi| \to \infty.
\]
On the other hand, smoothness implies that the above function is bounded
for $|\xi|\leq 1$. Hence it suffices to check integrability for
$|\xi|>1$.

Since $(1+|\xi|)$ and $|\xi|$ have comparable growth at infinity, we
obtain
\begin{align*}
\int_{|\xi|>1}
\big|
(1+|\xi|)^{ls}
\big(
|\widehat{f_{\alpha}}| *_{\xi} |\widehat{\chi}|
*_{\xi} |\widehat{\overline{g}}|(\xi)
\big)
\big|^{q}
\,\mathrm{d}\xi
\sim
\int_{|\xi|\geq 1}
|\xi|^{(ls-(d-\alpha))q}
\,\mathrm{d}\xi.
\end{align*}
The above estimate is finite only when
\[
q>\frac{d}{d-\alpha-ls}.
\]
Consequently, \(\chi f_{\alpha}\in \mathcal{M}^{p,q}_{s}\) for all
\(q>\frac{d}{d-\alpha-ls}\) provided $d-\alpha>ls$.

We next estimate the short-time Fourier transform of
\((1-\chi)f_{\alpha}\). Set
\[
h(x)=(1-\chi)f_{\alpha}(x).
\]
Observe that \(h(x)\simeq |x|^{-\alpha}\) for large \(|x|\). The
short-time Fourier transform of \(h\) is given by
\[
V_{g}h(x,\xi)
=\int_{\mathbb{R}^{d}} h(y)\, g(y-x)\, e^{-i\xi\cdot y}\,
\mathrm{d}y .
\]
Since
\[
(1-\Delta_{y})^{N} e^{-i\xi\cdot y}
=(1+|\xi|^{2})^{N} e^{-i\xi\cdot y}
\]
for any integer \(N\geq 0\), we may write
\begin{align*}
V_{g}h(x,\xi)
&=(1+|\xi|^{2})^{-N}
\int_{\mathbb{R}^{d}} h(y)\, g(y-x)\,
(1-\Delta_{y})^{N} e^{-i\xi\cdot y}\,
\mathrm{d}y .
\end{align*}
Applying integration by parts, we obtain
\begin{align*}
|V_{g}h(x,\xi)|
&\leq (1+|\xi|^{2})^{-N}
\int_{\mathbb{R}^{d}}
\bigl|(1-\Delta_{y})^{N}\bigl(h(y)g(y-x)\bigr)\bigr|
\, \mathrm{d}y .
\end{align*}
Now,
\begin{align*}
(1-\Delta_{y})^{N}\bigl(h(y)\, g(y-x)\bigr)
= \sum_{|\eta|\leq 2N} C_{\eta}\, \partial^{\eta} h(y)\, \Phi_{\eta}(y-x),
\end{align*}
where each \(\Phi_{\eta}\) is a finite linear combination of derivatives of
\(g\). Consequently, we obtain
\begin{align}
\label{boundofSTFT}
|V_{g}h(x,\xi)|
&\leq (1+|\xi|^{2})^{-N}
\sum_{|\eta|\leq 2N} C_{\eta}
\int_{\mathbb{R}^{d}}
|\partial^{\eta} h(y)|\, |\Phi_{\eta}(y-x)|\, \mathrm{d}y \nonumber\\
&\leq (1+|\xi|^{2})^{-N}
\sum_{|\eta|\leq 2N} C_{\eta}\, I_{\eta},
\end{align}
where
\[
I_{\eta}
= \int_{\mathbb{R}^{d}}
|\partial^{\eta} h(y)|\, |\Phi_{\eta}(y-x)|\, \mathrm{d}y .
\]

Note that \(h(y)=(1-\chi)f_{\alpha}\) is smooth on its support. Moreover, for
\(|y|>1\) we have
\[
|\partial^{\eta} h(y)|
\leq C_{\eta}\,(1+|y|)^{-\alpha-|\eta|}.
\]
Since \(\Phi_{\eta}\in \mathcal{S}(\mathbb{R}^{d})\), for any \(M\in\mathbb{N}\)
there exists \(C_{\eta}>0\) such that
\[
|\Phi_{\eta}(z)|
\leq C_{\eta}\,(1+|z|)^{-M},
\qquad z\in\mathbb{R}^{d}.
\]
The value of \(M\) will be chosen later. Performing the change of variables
\(z=y-x\), we obtain
\begin{align}
\label{boundofI}
I_{\eta}
\leq \int_{\mathbb{R}^{d}}
(1+|x+z|)^{-\alpha-|\eta|}\,(1+|z|)^{-M}\, \mathrm{d}z .
\end{align}

For \(|z|\leq \tfrac{|x|}{2}\), we have
\(|x+z|\geq |x|-|z|\geq \tfrac{|x|}{2}\), and hence
\[
(1+|x+z|)^{-\alpha-|\eta|}
\lesssim (1+|x|)^{-\alpha-|\eta|}.
\]
Using the preceding estimate and choosing \(M>d\), we obtain
\begin{align}
\label{boundofI1}
\int_{|z|\leq \frac{|x|}{2}}
(1+|x+z|)^{-\alpha-|\eta|}(1+|z|)^{-M}\, \mathrm{d}z
&\lesssim
\int_{|z|\leq \frac{|x|}{2}}
(1+|x|)^{-\alpha-|\eta|}(1+|z|)^{-M}\, \mathrm{d}z \nonumber\\
&\leq
(1+|x|)^{-\alpha-|\eta|}
\int_{\mathbb{R}^{d}} (1+|z|)^{-M}\, \mathrm{d}z \nonumber\\
&\lesssim (1+|x|)^{-\alpha-|\eta|} \nonumber\\
&\lesssim (1+|x|)^{-\alpha}.
\end{align}

For the complementary region \(|z|>\tfrac{|x|}{2}\), observe that
\[
2(1+|z|)\geq (1+|x|).
\]
We now fix \(M=M_{1}+M_{2}\) with \(M_{1}>\alpha\) and \(M_{2}>d\).
 With the above estimates, we bound the integral over the region
\(|z|>\tfrac{|x|}{2}\) as follows:
\begin{align*}
\int_{|z|>\frac{|x|}{2}}
(1+|x+z|)^{-\alpha-|\eta|}(1+|z|)^{-M}\, \mathrm{d}z
&\leq
\int_{|z|>\frac{|x|}{2}} (1+|z|)^{-M}\, \mathrm{d}z \\
&=
\int_{|z|>\frac{|x|}{2}}
(1+|z|)^{-M_{1}} (1+|z|)^{-M_{2}}\, \mathrm{d}z \\
&\lesssim
(1+|x|)^{-\alpha}
\int_{|z|>\frac{|x|}{2}} (1+|z|)^{-M_{2}}\, \mathrm{d}z \\
&\lesssim (1+|x|)^{-\alpha}\int_{\mathbb{R}^{d}}(1+|z|)^{-M_{2}}\ \mathrm{d}z\\
&\lesssim (1+|x|)^{-\alpha}.
\end{align*}

Consequently,
\begin{align}
\label{boundofI2}
\int_{|z|>\frac{|x|}{2}}
(1+|x+z|)^{-\alpha-|\eta|}(1+|z|)^{-M}\, \mathrm{d}z
\lesssim (1+|x|)^{-\alpha}.
\end{align}

Combining \eqref{boundofI1} and \eqref{boundofI2} with \eqref{boundofI}, we
conclude that
\[
I_{\eta}\lesssim (1+|x|)^{-\alpha}.
\]
Therefore, from \eqref{boundofSTFT} we obtain
\[
|V_{g}h(x,\xi)|
\lesssim_{N} (1+|\xi|)^{-N} (1+|x|)^{-\alpha}.
\]

Consequently, \((1-\chi)f_{\alpha}\in \mathcal{M}^{p,q}_{s}\) provided that
\[
\| V_{g}h(x,\xi)\,(q_{1}+V(x)^{1/2}+|\xi|^{l})^{s} \|_{L^{p}_{x}L^{q}_{\xi}} < \infty .
\]
To verify this, it suffices to show that
\[
\|(1+|x|)^{-\alpha} (1+|\xi|)^{s-N}(q_{1}+V(x)^{1/2}+|\xi|^{l})^{s}\|_{L^{p}_{x}L^{q}_{\xi}} < \infty .
\]
Note that as $V(x)\sim |x|^{2k}$ 
\begin{align*}
q_{1}+V(x)^{1/2}+|\xi|^{l}\sim q_{1}+|x|^{k}+|\xi|^{l}\lesssim (1+|x|)^{k}(1+|\xi|)^{l}.
\end{align*}
Therefore, 
\begin{align*}
\|V_{g}h(x,\xi)\,(q_{1}+V(x)^{1/2}+|\xi|^{l})^{s}\|_{L^{p}_{x}L^{q}_{\xi}}\lesssim \| (1+|x|)^{ks-\alpha}(1+|\xi|)^{ls-N}\|_{L^{p}_{x}L^{q}_{\xi}}.
\end{align*}
We first choose \(N\) sufficiently large so that
\[
\|(1+|\xi|)^{ls-N}\|_{L^{q}_{\xi}} < \infty .
\]
Moreover, 
\[
\|(1+|x|)^{ks-\alpha}\|_{L^{p}_{x}} < \infty
\quad \text{if and only if} \quad
p > \frac{d}{\alpha-ks}
\]
provided $\alpha>ks$. 
This completes the proof.
\end{proof}

To prove the well-posedness result, we recall a fundamental algebra
property of modulation spaces, which allows us to control multilinear terms
arising from the inhomogeneous nonlinearity in the space
\(\mathcal{M}^{p,q}_{s}\).
We refer to
\cite{feichtinger1983modulation, MR2204673, MR2506839}
for the proof and further discussion.

\begin{proposition}
\label{algebraproperty}
Let \(m\in\mathbb{N}\), \(m\geq 1\), and assume that
\[
\sum_{i=1}^{m} \frac{1}{p_{i}} = \frac{1}{p_{0}},
\qquad
\sum_{i=1}^{m} \frac{1}{q_{i}} = m-1 + \frac{1}{q_{0}},
\]
with \(0<p_{i}\leq \infty\) and \(1\leq q_{i}\leq \infty\) for
\(1\leq i\leq m\).
Then there exists a constant \(C>0\) such that
\[
\biggl\|\prod_{i=1}^{m} f_{i}\biggr\|_{\mathcal{M}^{p_{0},q_{0}}_{s}}
\leq C \prod_{i=1}^{m} \|f_{i}\|_{\mathcal{M}^{p_{i},q_{i}}_{s}}.
\]
\end{proposition}

As a consequence of Proposition~\ref{algebraproperty}, we obtain the following
multilinear estimate, which will be used to control the nonlinear term in the
well-posedness analysis.

\begin{lemma}
\label{multilinearlema}
Let \(f,g\in \mathcal{M}^{p,q}_{s}\) for some \(0<p\leq \infty\) and
\(1\leq q\leq \infty\). Then the following multilinear estimate holds:
\begin{align}
\label{multilinearestimate}
\| g\, |f|^{2\nu} f \|_{\mathcal{M}^{p,r}_{s}}
\leq
C\, \|g\|_{\mathcal{M}^{p,q}_{s}}\,
\|f\|^{2\nu+1}_{\mathcal{M}^{p,q}_{s}},
\end{align}
for \(1\leq p,q,r\leq \infty\), \(\nu\in\mathbb{N}\), and
\[
\frac{2\nu+2}{q} = \frac{1}{r} + 2\nu + 1 .
\]
\end{lemma}

\begin{proof}
Let \(p_{0}>0\) be such that
\[
\frac{2\nu+2}{p} = \frac{1}{p_{0}} .
\]
By Proposition~\ref{algebraproperty}, we obtain
\[
\| g\, |f|^{2\nu} f \|_{\mathcal{M}^{p_{0},r}_{s}}
\leq
C\, \|g\|_{\mathcal{M}^{p,q}_{s}}\,
\|f\|^{2\nu+1}_{\mathcal{M}^{p,q}_{s}} .
\]
Since \(p_{0}< p\) by construction, the embedding
\[
\mathcal{M}^{p_{0},r}_{s} \hookrightarrow \mathcal{M}^{p,r}_{s}
\]
holds. Consequently,
\begin{align*}
\| g\, |f|^{2\nu} f \|_{\mathcal{M}^{p,r}_{s}}
&\leq
\| g\, |f|^{2\nu} f \|_{\mathcal{M}^{p_{0},r}_{s}} \\
&\leq
C\, \|g\|_{\mathcal{M}^{p,q}_{s}}\,
\|f\|^{2\nu+1}_{\mathcal{M}^{p,q}_{s}},
\end{align*}
which proves \eqref{multilinearestimate}.
\end{proof}

We next state the following theorem, which establishes global existence, uniqueness, and continuous dependence for the inhomogeneous problem \eqref{inhomogeneousequation}. The proof relies essentially on Theorem~\ref{BoundofH}.
\begin{theorem}\label{Anharinhom}
Let $s>0$ be such that  \(ks<\alpha<d-ls\), and
\(1<p,q<\infty\) satisfy
\[
q' \geq 2\nu+2, \qquad
q' > \frac{(2\nu+1)d}{\beta l}, \qquad
q > \frac{d}{d-\alpha-ls}, \qquad
p > \frac{d}{\alpha-ks}.
\]
Then there exists \(R>0\) such that, for every initial data
\(u_{0}\in \mathcal{M}^{p,q}_{s}\) with
\[
\|u_{0}\|_{\mathcal{M}^{p,q}_{s}} \leq R,
\]
the inhomogeneous problem \eqref{inhomogeneousequation} admits a unique
global solution
\[
u \in L^{\infty}\bigl([0,\infty),\, \mathcal{M}^{p,q}_{s}\bigr).
\]
Moreover, if $1 \leq p,q < \infty$, then the solution depends continuously on time and satisfies
\[
u \in C\big([0,\infty), \mathcal{M}^{p,q}_{s}\big).
\]
\end{theorem}

\begin{proof}
By Duhamel’s principle, the integral formulation of
\eqref{inhomogeneousequation} is given by
\begin{align}
\label{DuhamelforH}
u(t)
= S(t)u_{0}
+ \lambda \int_{0}^{t}
S(t-\tau)\bigl[f_{\alpha}|u(\tau)|^{2\nu}u(\tau)\bigr]\,
\mathrm{d}\tau
=: \mathcal{J}(u)(t),
\end{align}
where \(S(t)=e^{-t\Hd^{\beta}}\).
Let \(p,q\) be as in the statement of the theorem.

By Theorem~\ref{BoundofH}, we have
\begin{align}
\label{boundofS}
\|S(t)u_{0}\|_{\mathcal{M}^{p,q}_{s}}
\leq C_{1}\|u_{0}\|_{\mathcal{M}^{p,q}_{s}},
\end{align}
for some constant \(C_{1}>0\).

Since \(q'\geq 2\nu+2\), we may choose \(1\leq r\leq \infty\) such that
\[
\frac{2\nu+2}{q}=\frac{1}{r}+2\nu+1 .
\]
Using Theorem~\ref{BoundofH} and applying Minkowski’s integral inequality together with
Lemma~\ref{multilinearlema}, we obtain
\begin{align}
\nonumber
&\Biggl\|
\int_{0}^{t}
S(t-\tau)\bigl[f_{\alpha}|u(\tau)|^{2\nu}u(\tau)\bigr]\,
\mathrm{d}\tau
\Biggr\|_{\mathcal{M}^{p,q}_{s}} \\
\nonumber
&\leq
\int_{0}^{t}
\bigl\|
S(t-\tau)\bigl[f_{\alpha}|u(\tau)|^{2\nu}u(\tau)\bigr]
\bigr\|_{\mathcal{M}^{p,q}_{s}}
\,\mathrm{d}\tau \\
\nonumber
&\leq
C'
\begin{cases}
\displaystyle
\int_{0}^{t}
(t-\tau)^{-\sigma}
\|f_{\alpha}|u(\tau)|^{2\nu}u(\tau)\|_{\mathcal{M}^{p,r}_{s}}
\,\mathrm{d}\tau,
& 0<t\leq 1, \\[3mm]
\displaystyle
\int_{0}^{t}
e^{-(t-\tau)d^{\beta}}
\|f_{\alpha}|u(\tau)|^{2\nu}u(\tau)\|_{\mathcal{M}^{p,r}_{s}}
\,\mathrm{d}\tau,
& t\geq 1 ,
\end{cases}
\\
\nonumber
&\leq
C''
\begin{cases}
\displaystyle
\int_{0}^{t}
(t-\tau)^{-\sigma}
\|f_{\alpha}\|_{\mathcal{M}^{p,q}_{s}}
\|u(\tau)\|^{2\nu+1}_{\mathcal{M}^{p,q}_{s}}
\,\mathrm{d}\tau,
& 0<t\leq 1, \\[3mm]
\displaystyle
\int_{0}^{t}
e^{-(t-\tau)d^{\beta}}
\|f_{\alpha}\|_{\mathcal{M}^{p,q}_{s}}
\|u(\tau)\|^{2\nu+1}_{\mathcal{M}^{p,q}_{s}}
\,\mathrm{d}\tau,
& t\geq 1 ,
\end{cases}
\\
\label{deltainequality}
&\leq
C_{2}
\begin{cases}
\displaystyle
\int_{0}^{t}
(t-\tau)^{-\sigma}
\|u(\tau)\|^{2\nu+1}_{\mathcal{M}^{p,q}_{s}}
\,\mathrm{d}\tau,
& 0<t\leq 1, \\[3mm]
\displaystyle
\int_{0}^{t}
e^{-(t-\tau)d^{\beta}}
\|u(\tau)\|^{2\nu+1}_{\mathcal{M}^{p,q}_{s}}
\,\mathrm{d}\tau,
& t\geq 1 .
\end{cases}
\end{align}
 The last inequality follows from the admissibility conditions on the exponents
together with Lemma~\ref{|x|inMpq}, which ensures that
\(f_{\alpha}(x)=|x|^{-\alpha}\in \mathcal{M}^{p,q}_{s}\).
Moreover, we recall that
\begin{align*}
\sigma
&= \frac{d}{2\beta l}
\left(\frac{1}{q}-\frac{1}{r}\right).
\end{align*}
By the choice of \(r\), we have
\[
\frac{1}{q}-\frac{1}{r}
= (2\nu+1)\Bigl(1-\frac{1}{q}\Bigr)
= \frac{2\nu+1}{q'} .
\]
Since \(q' > \frac{(2\nu+1)d}{\beta l} \), it follows that
\begin{align*}
\sigma
= \frac{d(2\nu+1)}{2 l\beta q'}
< 1 .
\end{align*}
Therefore, from \eqref{deltainequality} we deduce that
\begin{align*}
\Biggl\|
\int_{0}^{t}
S(t-\tau)\bigl[f_{\alpha}|u(\tau)|^{2\nu}u(\tau)\bigr]\,
\mathrm{d}\tau
\Biggr\|_{\mathcal{M}^{p,q}_{s}}
&\leq
C_{3}\,
\|u\|^{2\nu+1}_{L^{\infty}([0,t),\,\mathcal{M}^{p,q}_{s})}
\begin{cases}
\displaystyle
\int_{0}^{t} \tau^{-\sigma}\,\mathrm{d}\tau,
& 0<t\leq 1, \\[3mm]
\displaystyle
\int_{0}^{t} e^{-\tau d^{\beta}}\,\mathrm{d}\tau,
& t\geq 1 ,
\end{cases}
\\
&\leq
C_{4}\,
\|u\|^{2\nu+1}_{L^{\infty}([0,\infty),\,\mathcal{M}^{p,q}_{s})},
\end{align*}
where we used the fact that \(\sigma<1\).

Combining this estimate with \eqref{boundofS}, we obtain
\begin{align}
\label{boundofJ}
\|\mathcal{J}(u)\|_{L^{\infty}([0,\infty),\,\mathcal{M}^{p,q}_{s})}
\leq
C_{5}
\left(
\|u_{0}\|_{\mathcal{M}^{p,q}_{s}}
+
\|u\|^{2\nu+1}_{L^{\infty}([0,\infty),\,\mathcal{M}^{p,q}_{s})}
\right),
\end{align}
for some constant \(C_{5}>0\). 

For \(R>0\), define
\[
B_{R}
:= \bigl\{
u \in L^{\infty}\bigl([0,\infty),\, \mathcal{M}^{p,q}_{s}\bigr)
:\;
\|u\|_{L^{\infty}([0,\infty),\, \mathcal{M}^{p,q}_{s})}
\leq R
\bigr\},
\]
which is a closed ball centered at the origin in
\(L^{\infty}([0,\infty), \mathcal{M}^{p,q}_{s})\).

We show that, for a suitable choice of \(R>0\), the mapping \(\mathcal{J}\)
maps \(B_{R}\) into itself.
Assume that
\[
\|u_{0}\|_{\mathcal{M}^{p,q}_{s}} \leq \frac{R}{2C_{5}} .
\]
Then, by \eqref{boundofJ}, for any \(u\in B_{R}\) we have
\[
\|\mathcal{J}(u)\|_{L^{\infty}([0,\infty),\, \mathcal{M}^{p,q}_{s})}
\leq \frac{R}{2} + C_{5} R^{2\nu+1}.
\]
Since \(\nu>0\), we may choose \(R>0\) sufficiently small so that
\(R^{2\nu} \leq (2C_{5})^{-1}\), which implies
\[
\|\mathcal{J}(u)\|_{L^{\infty}([0,\infty),\, \mathcal{M}^{p,q}_{s})}
\leq R .
\]
Hence, \(\mathcal{J}(u)\in B_{R}\).

Next, using the identity
\[
|u|^{2\nu}u - |v|^{2\nu}v
= (u-v)|u|^{2\nu}
+ v\bigl(|u|^{2\nu} - |v|^{2\nu}\bigr),
\]
and arguing as above, we obtain
\[
\|\mathcal{J}(u)-\mathcal{J}(v)\|_{L^{\infty}([0,\infty),\, \mathcal{M}^{p,q}_{s})}
\leq \frac{1}{2}
\|u-v\|_{L^{\infty}([0,\infty),\, \mathcal{M}^{p,q}_{s})},
\]
possibly after reducing \(R\) further.

Therefore, by Banach’s contraction mapping principle, the operator
\(\mathcal{J}\) admits a unique fixed point in \(B_{R}\), which yields a
unique global solution to \eqref{inhomogeneousequationforHd}. Following the same line of argument as in Theorem \ref{firstwellposed}, one can also establish that the solution belongs to the space $C([0,\infty), \mathcal{M}^{p,q}_{s})$ also and  this completes the proof.
\end{proof}

As a particular case of the above theorem, we consider the equation
\eqref{inhomogeneousequationforHd} in the specific setting of the Hermite
operator
\[
H := -\Delta + |x|^{2}.
\]
This choice is motivated by the fact that the Hermite operator provides
a canonical and well-understood model for anharmonic oscillators, while
retaining the essential spectral and phase-space features of the
general operator $\Hd$.

It is worth noting that $H$ corresponds to the generalised anharmonic
oscillator $\Hd$ in the special case $k=1$ with potential
$V(x)=|x|^{2}$.

We consider the following inhomogeneous nonlinear equation associated
with fractional powers of the Hermite operator:
\begin{align}
\label{inhomogeneousequation}
\begin{cases}
\partial_{t}u(t,x) + H^{\beta}u(t,x)
= \lambda |x|^{-\alpha}\big(|u|^{2\nu}u\big)(t,x),\\
u(0,x) = u_{0}(x),
\end{cases}
\end{align}
for $(t,x)\in[0,\infty)\times\mathbb{R}^{d}$, where
$u(t,x)\in\mathbb{C}$, $\lambda=\pm1$, $\alpha>0$, and $\nu\in\mathbb{N}$.

As in the previous Theorem~\ref{Anharinhom}, the initial data are taken from weighted
modulation spaces $\mathcal{M}^{p,q}_{s}$. In the Hermite setting, the
natural choice of modulation weight is
\[
v_{s}(x,\xi) = (1+|x|+|\xi|)^{s}, \qquad s\in\mathbb{R}.
\]
For the following two results and throughout the next section, $\mathcal{M}^{p,q}_{s}$ will denote the weighted modulation space corresponding to the weight function $v_{s}(x,\xi)$ defined above.

The global well-posedness result is a direct consequence of estimates of the type established in Theorem~\ref{BoundofH}. The proof carries over to the Hermite operator with minimal and routine modifications. We next state the following theorem, which establishes the corresponding result for the Hermite operator.

\begin{theorem}
\label{BoundofHp}
Let \(\beta>0\),
\(0<p_{1},p_{2},q_{1},q_{2}\leq \infty\),
\(s_{1}>0\), and \(s_{2}\in\mathbb{R}\). Define
\[
\frac{1}{\widetilde{p}}
:= \max\Bigl\{\frac{1}{p_{2}}-\frac{1}{p_{1}},\,0\Bigr\},
\qquad
\frac{1}{\widetilde{q}}
:= \max\Bigl\{\frac{1}{q_{2}}-\frac{1}{q_{1}},\,0\Bigr\},
\]
and
\[
\sigma_{0}
:= \frac{d}{2\beta}
\Bigl(\frac{1}{\widetilde{p}}+\frac{1}{\widetilde{q}}\Bigr).
\]
Then, for \(0<t\leq 1\), the estimate
\begin{align}
\label{maininequalityp}
\|e^{-tH^{\beta}}f\|_{\mathcal{M}^{p_{2},q_{2}}_{s_{2}}}
\leq C(t)\,\|f\|_{\mathcal{M}^{p_{1},q_{1}}_{s_{1}}}
\end{align}
holds, where
\begin{align}
\label{firstconstantp}
C(t)=C_{0}\, t^{-\sigma_{0}}
\end{align}
for some constant \(C_{0}>0\).

In particular, if \(s_{1}=s_{2}=s\geq 0\), the bound can be refined to
\begin{align}
\label{secondconstantp}
C(t)=
\begin{cases}
C_{0}\, t^{-\sigma_{0}}, & 0<t\leq 1,\\[2mm]
C_{0}\, e^{-t d^{\beta}}, & t\geq 1.
\end{cases}
\end{align}
\end{theorem}

To this end, we establish the following result, which is of independent interest and may be useful in related contexts. The proof is omitted, as it follows closely the argument of Theorem~\ref{Anharinhom}.
\begin{theorem}
Let \(0<\alpha<d\), \(0<s<2\beta\), and
\(1\leq p,q<\infty\) satisfy
\[
q' \geq 2\nu+2, \qquad
q' > \frac{(2\nu+1)d}{\beta}, \qquad
q > \frac{d}{d-\alpha-s}, \qquad
p > \frac{d}{\alpha-s}.
\]
Then there exists \(R>0\) such that, for every initial data
\(u_{0}\in \mathcal{M}^{p,q}_{s}\) with
\[
\|u_{0}\|_{\mathcal{M}^{p,q}_{s}} \leq R,
\]
the inhomogeneous problem \eqref{inhomogeneousequation} admits a unique
global solution
\[
u \in L^{\infty}\bigl([0,\infty),\, \mathcal{M}^{p,q}_{s}\bigr).
\]
Additionally, when $1 \leq p,q < \infty$, the solution is continuous with respect to time and, in particular, satisfies
\[
u \in C\big([0,\infty); \mathcal{M}^{p,q}_{s}\big).
\]
\end{theorem}
\section{Boundedness of the Ornstein-Uhlenbeck Semigroup on Gaussian Modulation Spaces}

In this section, we investigate the boundedness properties of the
Ornstein-Uhlenbeck semigroup on a class of modulation spaces equipped with
Gaussian short-time Fourier transform, denoted by \(\mathcal{M}^{p,q}_{\gamma,s}\), which we refer
to as Gaussian modulation spaces.

Recall that the Ornstein-Uhlenbeck operator
\[
\mathbf{L} = -\Delta + 2x \cdot \nabla + d
\]
is a positive self-adjoint operator on the Gaussian Hilbert space
\[
L^{2}(\gamma) := L^{2}(\mathbb{R}^{d}, \gamma(x)\,\mathrm{d}x),
\qquad
\gamma(x)=\pi^{-d/2} e^{-|x|^{2}},
\]
see \cite{MR3966424}. Moreover, one has the identity
\[
\sum_{j=1}^{d} \partial_{j}^{*}\partial_{j}
= -\Delta + 2x\cdot\nabla,
\]
where \(\partial_{j}=\partial/\partial x_{j}\) and
\(\partial_{j}^{*}=2x_{j}-\partial_{j}\) denotes the adjoint of
\(\partial_{j}\) on \(L^{2}(\gamma)\).

For \(f\in\mathcal{S}(\mathbb{R}^{d})\), define the multiplication operator
\(M_{\gamma}\) by
\[
M_{\gamma}f(x)
:= \pi^{-\frac{d}{4}}e^{-\frac{|x|^{2}}{2}}f(x)
= \gamma(x)^{1/2} f(x),
\]
where \(\gamma(x)=\pi^{-d/2}e^{-|x|^{2}}\) is the Gaussian density.
The Hermite operator
\[
H=-\Delta+|x|^{2}
\]
and the Ornstein-Uhlenbeck operator \(\mathbf{L}\) are related through the
intertwining identity
\[
M_{\gamma}\,\mathbf{L}\,M_{\gamma}^{-1} = H.
\]

This relation allows us to define fractional powers of the
Ornstein-Uhlenbeck operator via the spectral calculus of the Hermite
operator. Recall that the fractional power \(H^{\beta}\), \(\beta>0\), admits
the spectral decomposition \cite{MR4628746}
\[
H^{\beta}
= \sum_{j=0}^{\infty} (2j+d)^{\beta} P_{j},
\qquad
P_{j}f
= \sum_{|\alpha|=j} \langle f, \Phi_{\alpha} \rangle \Phi_{\alpha},
\]
where \(\{\Phi_{\alpha}\}_{\alpha\in\mathbb{N}^{d}}\) denotes the family of
normalized Hermite functions, forming an orthonormal basis of
\(L^{2}(\mathbb{R}^{d})\).

The operator \(H^{\beta}\) is a densely defined, unbounded operator on
\(L^{2}(\mathbb{R}^{d})\), and the associated heat semigroup is given by
\[
e^{-tH^{\beta}}f
= \sum_{j=0}^{\infty} e^{-t(2j+d)^{\beta}} P_{j}f,
\qquad t>0.
\]
Using the identity \(M_{\gamma}\mathbf{L}M_{\gamma}^{-1}=H\), we define the
heat semigroup generated by the fractional Ornstein-Uhlenbeck operator
\(\mathbf{L}^{\beta}\) as
\begin{align*}
e^{-t\mathbf{L}^{\beta}}f
&= \sum_{j=0}^{\infty}
e^{-t(2j+d)^{\beta}}\, M_{\gamma}^{-1} P_{j} M_{\gamma} f \\
&= M_{\gamma}^{-1}
\sum_{j=0}^{\infty} e^{-t(2j+d)^{\beta}} P_{j} M_{\gamma} f \\
&= M_{\gamma}^{-1} e^{-tH^{\beta}} M_{\gamma} f .
\end{align*}
In particular, we obtain the intertwining relation
\begin{align}
\label{intertwiningOU}
M_{\gamma}^{-1} e^{-tH^{\beta}} M_{\gamma}
= e^{-t\mathbf{L}^{\beta}},
\end{align}
which will be the key tool in establishing boundedness results on
Gaussian modulation spaces.

We now introduce the Gaussian modulation spaces
\(\mathcal{M}^{p,q}_{\gamma,s}\).

For a tempered distribution \(f\in\mathcal{S}'(\mathbb{R}^{d})\), we define the
Gaussian short-time Fourier transform (Gaussian STFT) by
\begin{align*}
V^{\gamma}_{g} f(x,\xi)
:= \pi^{-\frac{d}{4}}\int_{\mathbb{R}^{d}}
f(z)\, \overline{g(z-x)}\, e^{-2\pi i \xi\cdot z}\,
e^{-\frac{|z|^{2}}{2}}\, \mathrm{d}z .
\end{align*}
Note that
\[
V^{\gamma}_{g} f(x,\xi)
= V_{g}\bigl(M_{\gamma} f\bigr)(x,\xi).
\]

The Gaussian modulation space norm is defined by
\begin{align*}
\|f\|_{\mathcal{M}^{p,q}_{\gamma,s}}
:= \left(
\int_{\mathbb{R}^{d}}
\left(
\int_{\mathbb{R}^{d}}
\bigl|V^{\gamma}_{g} f(x,\xi)\, v_{s}(x,\xi)\bigr|^{p}
\,\mathrm{d}x
\right)^{\frac{q}{p}}
\,\mathrm{d}\xi
\right)^{\frac{1}{q}},
\end{align*}
for \(0\leq p,q<\infty\), with the usual modifications when \(p=\infty\) or
\(q=\infty\).

\begin{definition}[Gaussian modulation space]
For \(0\leq p,q\leq \infty\), the Gaussian modulation space
\(\mathcal{M}^{p,q}_{\gamma,s}\) consists of all tempered distributions
\(f\in\mathcal{S}'(\mathbb{R}^{d})\) such that
\[
\|f\|_{\mathcal{M}^{p,q}_{\gamma,s}}<\infty.
\]
\end{definition}

It is immediate from the definition that
\[
\|f\|_{\mathcal{M}^{p,q}_{\gamma,s}}
= \|M_{\gamma}f\|_{\mathcal{M}^{p,q}_{s}} .
\]
Consequently, the operator
\[
M_{\gamma} : \mathcal{M}^{p,q}_{\gamma,s} \longrightarrow \mathcal{M}^{p,q}_{s}
\]
is an isometry. Similarly,
\[
M_{\gamma}^{-1} : \mathcal{M}^{p,q}_{s} \longrightarrow \mathcal{M}^{p,q}_{\gamma,s}
\]
is also an isometry. Hence, the spaces
\(\mathcal{M}^{p,q}_{\gamma,s}\) and \(\mathcal{M}^{p,q}_{s}\) are isometrically
isomorphic.

As a consequence, the boundedness of the Ornstein-Uhlenbeck semigroup
\(e^{-t\mathbf{L}^{\beta}}\) on \(\mathcal{M}^{p,q}_{\gamma,s}\) follows directly
from the boundedness of the Hermite semigroup \(e^{-tH^{\beta}}\) on
\(\mathcal{M}^{p,q}_{s}\).
Indeed, using the intertwining identity
\(e^{-t\mathbf{L}^{\beta}} = M_{\gamma}^{-1} e^{-tH^{\beta}} M_{\gamma}\), we have
\begin{align*}
\|e^{-t\mathbf{L}^{\beta}}f\|_{\mathcal{M}^{p,q}_{\gamma,s}}
&= \|M_{\gamma}^{-1} e^{-tH^{\beta}} M_{\gamma} f\|_{\mathcal{M}^{p,q}_{\gamma,s}} \\
&= \|e^{-tH^{\beta}} M_{\gamma} f\|_{\mathcal{M}^{p,q}_{s}} \\
&\leq C \|M_{\gamma} f\|_{\mathcal{M}^{p,q}_{s}} \\
&= C \|f\|_{\mathcal{M}^{p,q}_{\gamma,s}},
\end{align*}
where \(C>0\) is a constant and the inequality follows from the boundedness of
\(e^{-tH^{\beta}}\) on \(\mathcal{M}^{p,q}_{s}\).

We therefore obtain the following analogue of Theorem~\ref{BoundofH} in the
setting of Gaussian modulation spaces and for the
Ornstein-Uhlenbeck operator.

\begin{theorem}
\label{BoundofHp1}
Let \(\beta>0\),
\(0<p_{1},p_{2},q_{1},q_{2}\leq \infty\),
\(s_{1}>0\), and \(s_{2}\in\mathbb{R}\). Define
\[
\frac{1}{\widetilde{p}}
:= \max\Bigl\{\frac{1}{p_{2}}-\frac{1}{p_{1}},\,0\Bigr\},
\qquad
\frac{1}{\widetilde{q}}
:= \max\Bigl\{\frac{1}{q_{2}}-\frac{1}{q_{1}},\,0\Bigr\},
\]
and
\[
\sigma_{1}
:= \frac{d}{2\beta}
\Bigl(\frac{1}{\widetilde{p}}+\frac{1}{\widetilde{q}}\Bigr).
\]

Then, for \(0<t\leq 1\),
\begin{align}
\label{maininequalityp1}
\|e^{-t\mathbf{L}^{\beta}}f\|_{\mathcal{M}^{p_{2},q_{2}}_{\gamma,s_{2}}}
\leq C(t)\, \|f\|_{\mathcal{M}^{p_{1},q_{1}}_{\gamma,s_{1}}},
\end{align}
where
\begin{align}
\label{firstconstantp1}
C(t)=C_{0}\, t^{-\sigma_{1}}
\end{align}
for some constant \(C_{0}>0\).
In particular, if \(s_{1}=s_{2}=s\geq 0\), then
\begin{align}
\label{secondconstantp1}
C(t)=
\begin{cases}
C_{0}\, t^{-\sigma_{1}}, & 0<t\leq 1, \\[2mm]
C_{0}\, e^{-t d^{\beta}}, & t\geq 1 .
\end{cases}
\end{align}
\end{theorem}

We conclude this section by noting that, in view of the above results, the
well-posedness theory for evolution equations associated with the
Ornstein-Uhlenbeck operator can be developed within the framework of
Gaussian modulation spaces. Since the arguments are identical to
those presented in Section~4, we omit the details.

\section*{Concluding Remarks}

In this paper, we investigated the fractional powers of generalised anharmonic oscillator
\[
\mathcal{H}^{\beta}_{k,l}:= \left((-\Delta)^l + V(x)\right)^{\beta},
\]
where $\beta>0$ and $V(x)$ is a strictly positive homogeneous potential of polynomial growth of order $|x|^{2k}$. We analysed the associated semigroup $e^{-t\mathcal{H}_{k,l}^{\beta}}$ on weighted modulation spaces $\mathcal{M}^{p,q}_{s}$, $s\in\mathbb{R}$, for the full range $0<p,q\leq\infty$. In particular, we obtained estimates of the form
\[
\bigl\| e^{-t\mathcal{H}_{k,l}^{\beta}} f \bigr\|_{\mathcal{M}^{p_2,q_2}_{s_2}}
\le
C(t)\,
\|f\|_{\mathcal{M}^{p_1,q_1}_{s_1}},
\]
with explicit dependence of the constant $C(t)$ on time.
\begin{itemize}
\item First, by constructing a suitable Hörmander metric adapted to the operator $\mathcal{H}_{k,l}$ and employing phase-space techniques based on the short-time Fourier transform, we established boundedness properties of the semigroup $e^{-t\mathcal{H}_{k,l}^{\beta}}$ on weighted modulation spaces in both the short- and long-time regimes.
\item We showed that the obtained modulation space estimates encode pointwise information on the integral kernel of $e^{-t\mathcal{H}_{k,l}^{\beta}}$, yielding off-diagonal decay consistent with the anisotropic scaling induced by $(-\Delta)^l$ and the confining nature of the potential $V(x)$.
\item In the short-time regime, the kernel exhibits polynomial-type pointwise decay reflecting the underlying homogeneous structure of $\mathcal{H}_{k,l}$, while in the large-time regime the decay is governed by the lowest eigenvalue of $\mathcal{H}_{k,l}$, leading to pointwise exponential damping.
\end{itemize}
Furthermore, we applied these semigroup estimates to nonlinear evolution equations associated with $\mathcal{H}_{k,l}^{\beta}$.
\begin{itemize}
\item We established global existence results for the nonlinear heat equation
\[
\partial_t u + \mathcal{H}_{k,l}^{\beta} u =
\begin{cases}
\lambda |u|^{2\nu} u, & \text{(power-type nonlinearity)}, \\[6pt]
\lambda |x|^{-\alpha}\big(|u|^{2\nu}u\big), & \text{(spatially inhomogeneous power nonlinearity)}.
\end{cases}
\]
for initial data belonging to suitable weighted modulation spaces, exploiting the algebra properties of these spaces in the admissible range of parameters. Notably, no additional assumptions on $\beta$ and $\nu$ are needed in the present result and this follows entirely from the choice of the weight range for the parameter $s$ of the weighted modulation space.
As for spatially inhomogeneous power-type nonlinearities,the corresponding equation associated with classical Laplacian has been studied recently  in \cite{BHIMANI2026114106}. To the best of  our knowledge the analogous equation for the fractional anharmonic oscillator  has not yet been investigated. 
\item In the special case $k=l=1$, corresponding to the Hermite operator $H=-\Delta+|x|^{2}$, we obtained global-in-time solutions for an inhomogeneous nonlinear heat equation under appropriate assumptions on the parameters and the initial data.
\end{itemize}
Finally, we introduced a new class of Gaussian modulation spaces $\mathcal{M}^{p,q}_{\gamma,s}$ defined via the Gaussian short-time Fourier transform and studied the heat semigroup generated by the Ornstein-Uhlenbeck operator
\[
\mathbf{L} = -\Delta + 2x\cdot\nabla + d.
\]
\begin{itemize}
\item Within this framework, we established boundedness and decay estimates analogous to those obtained for $\mathcal{H}_{k,l}^{\beta}$, highlighting the role of Gaussian phase-space localization.

\item These results further demonstrate that the modulation space approach provides a unified phase-space framework that captures both pointwise kernel behavior and nonlinear dynamics for a broad class of diffusive semigroups associated with anharmonic-type operators.
\end{itemize}

\section*{Acknowledgements}
The second author acknowledges the financial assistance provided by the University Grants Commission (UGC), India (File No. 231610192540) during the course of the Ph.D. program.
\section*{Conflict of Interest}
The authors declare that there are no potential competing interests.
\bibliographystyle{abbrv}
\bibliography{ref}

\begin{thebibliography}{10}

\bibitem{MR2204673}
W.~Baoxiang, Z.~Lifeng, and G.~Boling.
\newblock Isometric decomposition operators, function spaces {$E^\lambda_{p,q}$} and applications to nonlinear evolution equations.
\newblock {\em J. Funct. Anal.}, 233(1):1--39, 2006.

\bibitem{BenyiOh12}
{\'A}.~B{\'e}nyi and T.~Oh.
\newblock Modulation spaces, wiener amalgam spaces, and nonlinear schr{\"o}dinger equations.
\newblock {\em Adv. Math.}, 228(5):2943--2981, 2011.

\bibitem{MR2506839}
A.~B\'enyi and K.~A. Okoudjou.
\newblock Local well-posedness of nonlinear dispersive equations on modulation spaces.
\newblock {\em Bull. Lond. Math. Soc.}, 41(3):549--558, 2009.

\bibitem{MR4286055}
A.~B\'enyi and K.~A. Okoudjou.
\newblock {\em Modulation spaces---with applications to pseudodifferential operators and nonlinear {S}chr\"odinger equations}.
\newblock Applied and Numerical Harmonic Analysis. Birkh\"auser/Springer, New York, [2020] \copyright 2020.

\bibitem{BHIMANI2026114106}
D.~G. Bhimani, D.~Dhingra, and V.~K. Sohani.
\newblock Low-regularity global solution of the inhomogeneous nonlinear schrödinger equations in modulation spaces.
\newblock {\em Journal of Differential Equations}, 458:114106, 2026.

\bibitem{MR4313961}
D.~G. Bhimani, R.~Manna, F.~Nicola, S.~Thangavelu, and S.~I. Trapasso.
\newblock Phase space analysis of the {H}ermite semigroup and applications to nonlinear global well-posedness.
\newblock {\em Adv. Math.}, 392:Paper No. 107995, 18, 2021.

\bibitem{MR4673072}
D.~G. Bhimani, R.~Manna, F.~Nicola, S.~Thangavelu, and S.~I. Trapasso.
\newblock On heat equations associated with fractional harmonic oscillators.
\newblock {\em Fract. Calc. Appl. Anal.}, 26(6):2470--2492, 2023.

\bibitem{MR1011988}
J.-M. Bony and N.~Lerner.
\newblock Quantification asymptotique et microlocalisations d'ordre sup\'erieur. {I}.
\newblock {\em Ann. Sci. \'Ecole Norm. Sup. (4)}, 22(3):377--433, 1989.

\bibitem{Boulkhemair98}
A.~Boulkhemair.
\newblock L$^2$ estimates for weyl quantization.
\newblock {\em J. Funct. Anal.}, 165(1):173--204, 1999.

\bibitem{MR4944933}
D.~Cardona, M.~Chatzakou, J.~Delgado, V.~Kumar, and M.~Ruzhansky.
\newblock Anharmonic semigroups and applications to global well-posedness of nonlinear heat equations.
\newblock {\em J. Math. Phys.}, 66(8):Paper No. 081509, 15, 2025.

\bibitem{MR4291583}
L.~Chaichenets and N.~Pattakos.
\newblock On the global wellposedness of the {K}lein-{G}ordon equation for initial data in modulation spaces.
\newblock {\em Proc. Amer. Math. Soc.}, 149(9):3849--3861, 2021.

\bibitem{MR4299820}
M.~Chatzakou, J.~Delgado, and M.~Ruzhansky.
\newblock On a class of anharmonic oscillators.
\newblock {\em J. Math. Pures Appl. (9)}, 153:1--29, 2021.

\bibitem{MR4489248}
M.~Chatzakou, J.~Delgado, and M.~Ruzhansky.
\newblock On a class of anharmonic oscillators {II}. {G}eneral case.
\newblock {\em Bull. Sci. Math.}, 180:Paper No. 103196, 22, 2022.

\bibitem{MR3554704}
P.~Chen, W.~Hebisch, and A.~Sikora.
\newblock Bochner-{R}iesz profile of anharmonic oscillator {$\mathcal{L}=-\frac{d^2}{dx^2}+|x|$}.
\newblock {\em J. Funct. Anal.}, 271(11):3186--3241, 2016.

\bibitem{MR4215324}
E.~Cordero.
\newblock On the local well-posedness of the nonlinear heat equation associated to the fractional {H}ermite operator in modulation spaces.
\newblock {\em J. Pseudo-Differ. Oper. Appl.}, 12(1):Paper No. 13, 13, 2021.

\bibitem{CorderoNicola07}
E.~Cordero and F.~Nicola.
\newblock Fourier integral operators in modulation spaces.
\newblock {\em J. Funct. Anal.}, 254(2):506--534, 2008.

\bibitem{CorderoNicola09}
E.~Cordero and F.~Nicola.
\newblock Some new embedding properties of modulation spaces.
\newblock {\em J. Funct. Anal.}, 256(11):3685--3711, 2009.

\bibitem{MR4201879}
E.~Cordero and L.~Rodino.
\newblock {\em Time-frequency analysis of operators}, volume~75 of {\em De Gruyter Studies in Mathematics}.
\newblock De Gruyter, Berlin, [2020] \copyright 2020.

\bibitem{CunananOkoudjou11}
M.~Cunanan and K.~A. Okoudjou.
\newblock Well-posedness of nonlinear schr{\"o}dinger equations in modulation spaces.
\newblock {\em J. Math. Anal. Appl.}, 384(2):759--772, 2011.

\bibitem{MR883643}
H.~L. Cycon, R.~G. Froese, W.~Kirsch, and B.~Simon.
\newblock {\em Schr\"odinger operators with application to quantum mechanics and global geometry}.
\newblock Texts and Monographs in Physics. Springer-Verlag, Berlin, study edition, 1987.

\bibitem{feichtinger1983modulation}
H.~G. Feichtinger.
\newblock {\em Modulation spaces on locally compact abelian groups}.
\newblock Universit{\"a}t Wien. Mathematisches Institut, 1983.

\bibitem{MR4849356}
H.~G. Feichtinger, M.~Kobayashi, and E.~Sato.
\newblock On some properties of modulation spaces as {B}anach algebras.
\newblock {\em Studia Math.}, 280(1):55--86, 2025.

\bibitem{FeichtingerNarimani10}
H.~G. Feichtinger and G.~Narimani.
\newblock Windowed fourier transform and modulation spaces.
\newblock {\em J. Fourier Anal. Appl.}, 16(4):495--528, 2010.

\bibitem{MR1681103}
I.~D. Feranchuk and A.~L. Tolstik.
\newblock Operator method for coupled anharmonic oscillators.
\newblock {\em J. Phys. A}, 32(11):2115--2128, 1999.

\bibitem{MR2028532}
Y.~V. Galperin and S.~Samarah.
\newblock Time-frequency analysis on modulation spaces {$M^{p,q}_m$}, {$0<p,\ q\leq\infty$}.
\newblock {\em Appl. Comput. Harmon. Anal.}, 16(1):1--18, 2004.

\bibitem{MR4628746}
P.~Ganguly, R.~Manna, and S.~Thangavelu.
\newblock An extension problem, trace {H}ardy and {H}ardy's inequalities for the {O}rnstein-{U}hlenbeck operator.
\newblock {\em Anal. PDE}, 16(5):1205--1244, 2023.

\bibitem{MR1843717}
K.~Gr\"ochenig.
\newblock {\em Foundations of time-frequency analysis}.
\newblock Applied and Numerical Harmonic Analysis. Birkh\"auser Boston, Inc., Boston, MA, 2001.

\bibitem{GuoFanWu11}
W.~Guo, D.~Fan, and H.~Wu.
\newblock Nonlinear evolution equations in modulation spaces.
\newblock {\em J. Math. Anal. Appl.}, 384(2):759--772, 2011.

\bibitem{MR743094}
B.~Helffer.
\newblock {\em Th\'eorie spectrale pour des op\'erateurs globalement elliptiques}, volume 112 of {\em Ast\'erisque}.
\newblock Soci\'et\'e{} Math\'ematique de France, Paris, 1984.
\newblock With an English summary.

\bibitem{MR657970}
B.~Helffer and D.~Robert.
\newblock Comportement asymptotique pr\'ecise du spectre d'op\'erateurs globalement elliptiques dans {${\bf R}\sp{n}$}.
\newblock In {\em Goulaouic-{M}eyer-{S}chwartz {S}eminar, 1980--1981}, pages Exp. No. II, 23. \'Ecole Polytech., Palaiseau, 1981.

\bibitem{MR683006}
B.~Helffer and D.~Robert.
\newblock Asymptotique des niveaux d'\'energie pour des hamiltoniens \`a{} un degr\'e{} de libert\'e.
\newblock {\em Duke Math. J.}, 49(4):853--868, 1982.

\bibitem{MR662451}
B.~Helffer and D.~Robert.
\newblock Propri\'et\'es asymptotiques du spectre d'op\'erateurs pseudodiff\'erentiels sur {${\bf R}\sp{n}$}.
\newblock {\em Comm. Partial Differential Equations}, 7(7):795--882, 1982.

\bibitem{MR2304165}
L.~H\"ormander.
\newblock {\em The analysis of linear partial differential operators. {III}}.
\newblock Classics in Mathematics. Springer, Berlin, 2007.
\newblock Pseudo-differential operators, Reprint of the 1994 edition.

\bibitem{MR2885963}
K.~Ishige and Y.~Kabeya.
\newblock {$L^p$} norms of nonnegative {S}chr\"odinger heat semigroup and the large time behavior of hot spots.
\newblock {\em J. Funct. Anal.}, 262(6):2695--2733, 2012.

\bibitem{MR2599384}
N.~Lerner.
\newblock {\em Metrics on the phase space and non-selfadjoint pseudo-differential operators}, volume~3 of {\em Pseudo-Differential Operators. Theory and Applications}.
\newblock Birkh\"auser Verlag, Basel, 2010.

\bibitem{liverts2008approximate}
E.~Liverts and V.~Mandelzweig.
\newblock Approximate analytic solutions of the schr{\"o}dinger equation for the generalized anharmonic oscillator.
\newblock {\em Physica Scripta}, 77(2):025003, 2008.

\bibitem{MR2668420}
F.~Nicola and L.~Rodino.
\newblock {\em Global pseudo-differential calculus on {E}uclidean spaces}, volume~4 of {\em Pseudo-Differential Operators. Theory and Applications}.
\newblock Birkh\"auser Verlag, Basel, 2010.

\bibitem{RodinoWahlberg12}
L.~Rodino and P.~Wahlberg.
\newblock The gabor wave front set.
\newblock {\em Monatsh. Math.}, 173(4):625--655, 2014.

\bibitem{MR3014810}
M.~Ruzhansky, M.~Sugimoto, and B.~Wang.
\newblock Modulation spaces and nonlinear evolution equations.
\newblock In {\em Evolution equations of hyperbolic and {S}chr\"odinger type}, volume 301 of {\em Progr. Math.}, pages 267--283. Birkh\"auser/Springer Basel AG, Basel, 2012.

\bibitem{MR752806}
B.~Simon.
\newblock Erratum: ``{S}chr\"odinger semigroups''.
\newblock {\em Bull. Amer. Math. Soc. (N.S.)}, 11(2):426, 1984.

\bibitem{SugimotoTomita07}
M.~Sugimoto and N.~Tomita.
\newblock The dilation property of modulation spaces and their inclusion relations.
\newblock {\em J. Funct. Anal.}, 248(1):79--106, 2007.

\bibitem{SugimotoTomita09}
M.~Sugimoto and N.~Tomita.
\newblock The dilation property of modulation spaces and their inclusion relations.
\newblock {\em J. Funct. Anal.}, 248(1):79--106, 2007.

\bibitem{Toft04}
J.~Toft.
\newblock Continuity properties for modulation spaces.
\newblock {\em J. Funct. Anal.}, 207(2):399--429, 2004.

\bibitem{MR3636061}
J.~Toft.
\newblock Continuity and compactness for pseudo-differential operators with symbols in quasi-{B}anach spaces or {H}\"ormander classes.
\newblock {\em Anal. Appl. (Singap.)}, 15(3):353--389, 2017.

\bibitem{MR3966424}
W.~Urbina-Romero.
\newblock {\em Gaussian harmonic analysis}.
\newblock Springer Monographs in Mathematics. Springer, Cham, 2019.
\newblock With a foreword by Sundaram Thangavelu.

\bibitem{Wahlberg10}
P.~Wahlberg.
\newblock Boundedness of pseudo-differential operators on modulation spaces.
\newblock {\em Math. Nachr.}, 283(7):1085--1098, 2010.

\bibitem{WangHudzik07}
B.~Wang and H.~Hudzik.
\newblock The global cauchy problem for the nls and nlkg with small rough data.
\newblock {\em J. Differential Equations}, 232(1):36--73, 2007.

\bibitem{WangZhao09}
B.~Wang and L.~Zhao.
\newblock Global well-posedness and scattering for nonlinear schr{\"o}dinger equations in modulation spaces.
\newblock {\em J. Differential Equations}, 246(4):1648--1681, 2009.

\end{thebibliography}

\end{document}